\documentclass[11pt]{article}

\usepackage{amsmath,amssymb,amsfonts}
\usepackage{amsthm}
\usepackage[pagewise]{lineno}

\usepackage{lmodern}
\numberwithin{equation}{section}
\newtheorem{theorem}{Theorem}[section]
\newtheorem{definition}{Definition}[section]
\newtheorem{proposition}{Proposition}[section]
\newtheorem{lemma}{Lemma}[section]
\newtheorem{remark}{Remark}[section]

\begin{document}

\author{Qiang Gao$^{1}$, Xiaoyan Zhang$^{1,}$\footnote{\small  Corresponding Author. Email: zxysd@sdu.edu.cn} \\
{\small $^1$ School of Mathematics, Shandong University, Jinan 250100, P.R. China}
}

\title{Normalized solutions of quasilinear Schr\"odinger equations in the general $L^2$-supercritical case}
\date{}
\maketitle

\begin{abstract}
This paper is devoted to studying the existence of normalized solutions for the following quasilinear Schr\"odinger equation
\begin{equation*}
\begin{aligned}
   -\Delta u-u\Delta u^2 +\lambda u=h(u)   \quad\mathrm{in}\ \mathbb{R}^{3},
\end{aligned}
\end{equation*}
where $\lambda$ appears as a Lagrange multiplier, $h$ is a $L^2$-supercritical and Sobolev subcritical nonlinearity. The solutions correspond to critical points of the energy functional subject to  the $L^2$-norm constraint $\int_{\mathbb{R}^3}|u|^2dx=a^2>0$. Taking into account the Pohozaev manifold and
perturbation method, we obtain the existence of ground state normalized solutions and infinitely many normalized solutions. Moreover, our results cover several relevant existing results in \cite{LZ2023}. And in the end, we get the asymptotic properties of energy as $a$ tends to $+\infty$ and $a$ tends to $0^+$.
\end{abstract}



\section{Introduction}\label{intro}
In this paper, we are concerned with the following quasilinear Schr\"odinger equation
\begin{equation}\label{eq01}
 -\Delta u-u\Delta u^2 +\lambda u=h(u)   \quad\mathrm{in}\ \mathbb{R}^{3},
\end{equation}
under the constraint
\begin{equation}\label{eq02}
\int_{\mathbb{R}^3}|u|^2dx=a^2,
\end{equation}
where $a>0$, $\lambda\in \mathbb{R}$ is a Lagrange multiplier.

The interest in studying \eqref{eq01}-\eqref{eq02} comes from seeking the standing wave of the following time-dependent quasilinear Schr\"odinger equation
\begin{equation}\label{eq03}
i\partial_{t}\psi=\Delta\psi+ \psi \Delta(|\psi|^2)+ g(|\psi|)\psi,
\end{equation}
where $i$ denotes the imaginary unit, $\psi=\psi(t,x) \in \mathbb{C}$ is the wave function, $g$ is an approprite nonlinearity. Equations of the form \eqref{eq03} have been involved in models of superfluid films in fluid mechanics and plasma physics \cite{K1981, LS1978, PG1976}. And the equations play an important role in dissipative quantum mechanics, condensed matter theory, the theory of Heisenberg ferromagnets and magnons. As for more details on physical background of \eqref{eq03}, we refer the readers to \cite{BN1990,H1980,KIK1990,MF1984,QC1982}.

Throughout the paper, we consider standing  wave solutions to \eqref{eq03}, which are solutions of the form
$$\psi(t,x)=e^{-i\lambda t}u(x),$$
where $u \in H^1(\mathbb{R}^3)$, the frequency $\lambda \in \mathbb{R}$, then \eqref{eq03} can be transformed to \eqref{eq01} with $h(u)=g(|u|)u$.

Generally speaking, there are two  different points  concerning the form of standing wave solutions: one can either prescribe the frequency $\lambda$ or the $L^2$-norm $\|u\|_2$. In contrast with the search of solutions to \eqref{eq03} when the frequency $\lambda \in \mathbb{R}$ is a prior given, the search of normalized solutions becomes more complex. In the normalized setting, it is natural to prescribe the value of the mass $\int_{\mathbb{R}^N}|u|^2dx$ so that $\lambda$ can be interpreted as the Lagrange multiplier. Naturally, for $h(t)=|t|^{p-2}t$, a new critical exponent appears, the $L^2$-critical exponent (also named mass-critical exponent): $p:=4+\frac{4}{N}$. It is the threshold exponent for many dynamical properties, such as global existence vs. blow-up, and the stability or instability of ground states. For further clarification, we agree that $L^2$-subcritical case and $L^2$-supercritical case mean that $p<4+\frac{4}{N}$ and $p>4+\frac{4}{N}$, respectively. Alternatively, the mass often admits a clear physical meaning: it represents the power supply in nonlinear optics, or the total number of atoms in Bose-Einstein condensation. They are two main fields of application of the NLS.

It is well known that the solutions for \eqref{eq01} admitting prescribed $L^2$-norm \eqref{eq02} are the critical points of the energy functional
$$I(u):=\frac{1}{2} \int_{\mathbb{R}^3} |\nabla u|^2 + \int_{\mathbb{R}^3} |u|^2 |\nabla u|^2 -\int_{\mathbb{R}^3} H(u) $$
restricted on $\mathcal{S}'(a):=\left\{u \in H^1(\mathbb{R}^3)\;| \;\int_{\mathbb{R}^3}|u|^2|\nabla u|^2 <+\infty, \int_{\mathbb{R}^3}|u|^2=a^2\right\}$, where $H(t):=\int_0^th(s)ds$. At this time, the frequency $\lambda$ is an unknown number that can be determined as the Lagrange multiplier associated with the constraint $\mathcal{S}'(a)$. In addition, it is quite meaningful to study normalized solutions. This is not only because mass is conserved along the trajectories of \eqref{eq03}, i.e.,
$$\int_{\mathbb{R}^3}|\psi(t,x)|^2 dx= \int_{\mathbb{R}^3}|\psi(0,x)|^2 dx$$
for all $t>0$, but also it can provide a good insight into the dynamical properties (such as, orbital stability and instability) of solutions to the equation \eqref{eq03}.

We introduce some results about the existence of normalized solutions. Jeanjean \cite{J1997} considered the semilinear Schr\"odinger equation
\begin{equation}\label{eq1.4}
-\triangle u +\lambda u =g(u),\;\;\; x \in \mathbb{R}^N,
\end{equation}
where $N \geq 1$, $\lambda \in \mathbb{R}$. By using a minimax procedure, Jeanjean showed that for each $a>0$, \eqref{eq1.4} possesses at least one couple $(u_a,\lambda_a)$ of weak
solution with $\|u_a\|_2=a$ for $N\geq 2$. At the same time, he obtained the existence of ground states for $N\geq 1$. But, afterward, there was little progress about the study of normalized solutions for
a long time. One of the main reasons is that it is hard to prove the boundedness of
the constrained Palais-Smale sequence when the functional is unbounded from below on
the constraint manifold. More recently, problems of such type begun to receive much
attention. By virtue of a fountain theorem type argument, Bartsch and de Valeriola \cite{Bd2013} got a multiplicity result of \eqref{eq1.4} with $\|u\|_2=a>0$. Soave \cite{S2020} studied the existence and
properties of ground states to the nonlinear Schr\"odinger equation with combined
power nonlinearities. For more results, we refer the readers to \cite{LZ2023-2, LZ2024, ZhangZhitao, N2020} and their references therein.

Compared to \eqref{eq1.4} where the term $u\Delta u^2$ is not present, the search of solutions of \eqref{eq01} and \eqref{eq02} presents a major difficulty. The functional associated with the quasilinear term $\int_{\mathbb{R}^3} |u|^2 |\nabla u|^2 $
is nondifferentiable in $\{u \in H^1(\mathbb{R}^3)\;| \;\int_{\mathbb{R}^3}|u|^2|\nabla u|^2 <+\infty \}$. Since the parameter $\lambda$ is unknown and $\|u\|_2$ is equal to a constant, we find that Nehari manifold approach \cite{LWW2004, LLW2013} and changing variables \cite{CJ2004, LWW2003} are no longer applicable. To overcome this difficulty, various arguments have been developed. One of the most important tasks is that Jeanjean, Luo and Wang \cite{JLW2015} introduced the perturbation method to deal with the normalized solutions of quasilinear Schr\"odinger equation.

As far as I know, there are relatively few results on \eqref{eq01} and \eqref{eq02}. In  \cite{CJS2010, JL2013}, the authors studied the minimization problem $\tilde{m}(a)=\inf_{u \in \mathcal{S}'(a)}\left\{\frac{1}{2} \int_{\mathbb{R}^N} |\nabla u|^2 + \int_{\mathbb{R}^N} |u|^2 |\nabla u|^2 -\frac{1}{p}\int_{\mathbb{R}^N} |u|^p\right\}$ with $2<p\leq 4+\frac{4}{N}$. In addition, the authors \cite{ZZ2018} considered the existence and asymptotic behavior of the minimizers to $\tilde{m}(a)=\inf_{u \in \mathcal{S}'(a)}\left\{\frac{1}{2} \int_{\mathbb{R}^N} \left(|\nabla u|^2 + V(x)|u|^2\right)+\int_{\mathbb{R}^N} |u|^2 |\nabla u|^2 -\frac{1}{p}\int_{\mathbb{R}^N} |u|^p\right\}$ with $2<p\leq 4+\frac{4}{N}$, where $V(x)$ is an infinite potential well. In a word, most of the results on the normalized solution of \eqref{eq01} and \eqref{eq02} have been related to $L^2$-subcritical case and $L^2$-critical case. Specially, Li and Zou \cite{LZ2023} have considered  $L^2$-supcritical case with $h(t)=|t|^{p-2}t, \ p>4+\frac{4}{N}$, Mao and Lu \cite{2024New} have studied combination case of $L^2$-subcritical and $L^2$-supercritical
nonlinearities with $h(t)=\tau|t|^{q-2}t+|t|^{p-2}t$, $\tau>0$, $2<q<2+\frac{4}{N}$ and $p>4+\frac{4}{N}$. Recently, Jeanjean, Zhang and Zhong \cite{ZZJ2025} have given a Schwarz symmetric ground state solution for $h(t)=|t|^{p-2}t$ and $p>4+\frac{4}{N}$. For more results, we refer the readers to \cite{HW2026, JY2026} and their references therein.

In this paper, we set $N=3$, take the general $h$ and make the following assumptions on $h$:   \\
$(h_1)\;  h \in \mathcal{C}(\mathbb{R},\mathbb{R}),h(-t)=-h(t)$;   \\
$(h_2)\;  \lim_{t \rightarrow 0} \frac{h(t)}{|t|^{\frac{13}{3}}} =0$;  \\
$(h_3)\;  \lim_{|t| \rightarrow \infty} \frac{h(t)}{|t|^{5}} =0$;   \\
$(h_4)\;  \lim_{|t| \rightarrow \infty} \frac{H(t)}{|t|^{\frac{16}{3}}} =+\infty $;    \\
$(h_5) \; \tilde{h}(t)t>\frac{16}{3}\tilde{H}(t)$ for any $t\neq0$, where $\tilde{H}(t):=h(t)t-2H(t) \in \mathcal{C}^1(\mathbb{R},\mathbb{R})$, $\tilde{h}(t):=\tilde{H}'(t)$;   \\
$(h_6)\;  h(t)t < 6H(t)$ for any $t\neq 0$;   \\
$(h_7)\;  \lim_{|t| \rightarrow 0} \frac{h(t)t}{|t|^{6}} =+\infty.$

\begin{remark}
   During the proof, we find that $\tilde{h}(t)t>p\tilde{H}(t)$ with $p \in \left(\frac{16}{3},6\right)$ is easy for us to handle. However, the coerciveness of energy functional becomes more cumbersome  for $p=\frac{16}{3}$ in assumption $(h_5)$.
\end{remark}

\begin{remark}
(1) It follows from $(h_5)$ and
$$\left[\frac{\tilde{H}(t)}{|t|^{\frac{16}{3}}}\right]'=\frac{\tilde{h}(t)t-\frac{16}{3}\tilde{H}(t)} {|t|^{\frac{19}{3}}} {\rm sgn}(t)$$
that:

$(h_5')\;  t \mapsto \frac{\tilde{H}(t)}{|t|^{\frac{16}{3}}}$ is strictly decreasing on $(-\infty,0)$ and strictly increasing on $(0,+\infty) $.

(2) We point out that the function $h(t)=|t|^{p-2}t$ with $\frac{16}{3} <p<6$ satisfies the above conditions $(h_1)-(h_7)$.
\end{remark}
We are now in a position to state our main results.

\begin{theorem}\label{Thm1.1}
Under the hypotheses $(h_1)-(h_6)$, equations \eqref{eq01} and \eqref{eq02} admit a radially symmetric positive ground state normalized solution $u \in H^1(\mathbb{R}^3)\cap L^{\infty}(\mathbb{R}^3)$.
\end{theorem}

\begin{remark}
    Both Theorem 1.3 (vi) in \cite{DSZZ} and Theorem \ref{Thm1.1} in our paper establish the existence of positive radially symmetric normalized solutions of quasilinear Schr\"odinger equations in the general $L^2$-supercritical case. However, the differences are also evident:  \\
	(1) The normalized solution in Theorem \ref{Thm1.1} of our paper is, in particular, a ground state solution; whereas Theorem 1.3 (vi) in \cite{DSZZ} does not further obtain such a special character of the solution.\\
	(2) The proof strategies of the two results are completely different. Theorem \ref{Thm1.1} in our paper mainly employs a perturbation method to overcome the non-differentiability of the quasilinear term, and the proof is carried out directly for the quasilinear equation itself. In contrast, Theorem 1.3 (vi) in \cite{DSZZ} proves the result by transforming the quasilinear equation into a semilinear one.\\
	(3) Apart from this single theorem in dispute, the contents of the two papers are entirely different: our paper also establishes the existence of infinitely many normalized solutions in the mass supercritical case and derives the asymptotic properties of the energy, none of which are addressed in \cite{DSZZ}.

\end{remark}

\begin{theorem}\label{Thm1.2}
Under the hypotheses $(h_1)-(h_6)$, equations \eqref{eq01} and \eqref{eq02} admit a sequence of normalized solutions $u^j \in H^1(\mathbb{R}^3)\cap L^{\infty}(\mathbb{R}^3) $ with increasing energy $I(u^j)\rightarrow +\infty$.
\end{theorem}

\begin{remark}
    The general form of $h$ causes greater difficulties than the special form $h(t)=|t|^{p-2}t$ in \cite{LZ2023}. We need to explore more technical results to discuss the geometry and coerciveness of energy functional. And the process of obtaining properties about the Pohozaev manifold is also more complicated.
\end{remark}

It is easy to see that if $u$ is the critical point of $I |_{\mathcal{S}(a)}$, then $u$ satisfies the following Pohozaev identity
$$P (u):=\int_{\mathbb{R}^3} |\nabla u|^2 +
5\int_{\mathbb{R}^3} |u|^2 |\nabla u|^2 -\frac{3}{2} \int_{\mathbb{R}^3} \tilde{H}(u).$$
We define
 $$m(a):=\inf_{u \in \mathcal{P}'(a)} I(u),$$
 $$\sigma(a):=\inf_{u \in \mathcal{S}'(a), \Theta(u,\lambda)=0} I(u),$$
  where $\mathcal{P}'(a):=\{u\in \mathcal{S}'(a)\ |\ P(u)=0 \}$, $\Theta(u,\lambda):=  -\Delta u-u\Delta u^2 -h(u) +\lambda u$.
 And we have the following results:
\begin{theorem}\label{Thm1.5}
    Let $(h_1)-(h_7)$ hold, then for the function $a \mapsto m(a)$ we have  \\
    (1) $m(a)$ is positive and lower semicontinuous;  \\
    (2) $m(a) \rightarrow 0^+$ as $a \rightarrow +\infty$;  \\
    (3) $m(a) \rightarrow +\infty$ as $a \rightarrow 0^+$; moreover, $\sigma(a) \rightarrow +\infty$ as $a \rightarrow 0^+$.
\end{theorem}

\begin{remark}
    As far as we know, no one has studied asymptotic properties of $m(a)$ and $\sigma(a)$ in $L^2$-supercritical case. At the same time, since the quasilinear term is non-differential in $H^1(\mathbb{R}^3)$,  we can only obtain that $m(a)$ is less than $\sigma(a)$.
\end{remark}

    We complete the introduction by sketching the structure of the paper. In Section 2, we define the perturbed functional, discuss the geometry of $I_\mu$ and give the properties of the Pohozaev manifold. In Section 3, we define the closed sublevel, minimax
level and obtain the Palais-Smale sequence. In Section 4, we study the convergence of the critical points of the perturbed functional as $\mu$ tends 0, moreover, we finish the proof of Theorem \ref{Thm1.1}. In Section 5,  we give the proof of Theorem \ref{Thm1.2}. In Section 6, we study the lower semi-continuity of  $m(a)$ and the asymptotic properties of  $m(a)$ and $\sigma(a)$, furthermore, we complete the proof of Theorem \ref{Thm1.5}.
 \bigskip

\noindent \textbf{Notation.}
Throughout this paper, we make use of the following notations.  \\
$\bullet$ The standard norm in $L^s(\mathbb{R}^3) (1\leq s\leq \infty)$ is denoted by $\|\cdot\|_s$.  \\
$\bullet$ $o_n(1)$ means a quantity which tends to 0 as $n$ tends to $\infty$.   \\
$\bullet$ The symbols $\rightarrow$ and $\rightharpoonup$ denote the strong and the weak convergence, respectively.  \\
$\bullet$ C, $C_\varepsilon, C_q$ stand for various constants whose exact values are irrelevant.

\section{Preliminaries}\label{preliminary}

In this section, we establish some useful preliminary results.
First of all, we take the perturbation method to deal with nondifferentiability, which has been applied to constrained situations in \cite{JLW2015}.

For $\mu \in (0,1]$ and $\theta \in \left(\frac{12}{5},3\right)$, we consider the functional $I_\mu: \Upsilon \mapsto \mathbb{R} $
\begin{equation} \label{eq2-1}
I_{\mu}(u):=\frac{\mu}{\theta} \int_{\mathbb{R}^3} |\nabla u|^{\theta}+I(u),
\end{equation}
where $\Upsilon =W^{1,\theta}(\mathbb{R}^3)\cap H^1(\mathbb{R}^3)$. It is clear that $I_\mu \in \mathcal{C}^1(\Upsilon)$. At the same time, we define
\begin{equation} \label{eq2-2}
\mathcal{S}(a):=\left\{u \in \Upsilon \;\bigg|\;\int_{\mathbb{R}^3}|u|^2dx=a^2\right\}.
\end{equation}

It is easy to see that if $u$ is the critical point of $I_\mu |_{\mathcal{S}(a)}$, then $u$ satisfies the following Pohozaev identity
$$P_\mu (u):=\mu(1+\gamma_\theta) \int_{\mathbb{R}^3} |\nabla u|^{\theta} +\int_{\mathbb{R}^3} |\nabla u|^2 +
5\int_{\mathbb{R}^3} |u|^2 |\nabla u|^2 -\frac{3}{2} \int_{\mathbb{R}^3} \tilde{H}(u) ,$$
where $\gamma_\theta=\frac{3(\theta-2)}{2\theta}$.

By $(h_1)-(h_3)$, for any $\varepsilon >0$, there exists a constant $C_\varepsilon >0$ such that

\begin{equation} \label{eq2-3}
|h(t)|<\varepsilon |t|^{\frac{13}{3}} +C_\varepsilon |t|^{5}
\end{equation}
and

\begin{equation} \label{eq2-4}
|h(t)|< \varepsilon|t|^{5} + C_\varepsilon |t|^{\frac{13}{3}}.
\end{equation}
By the Gagliardo-Nirenberg inequality \cite{A2008},
for any $u \in \{u \in L^2(\mathbb{R}^3)\;| \;\nabla (u^2) \in L^2(\mathbb{R}^3) \}$, there exists a constant $C_q >0$ such that
 \begin{equation}\label{GNI}
 \int_{\mathbb{R}^3} |u|^q \leq C_q\left (\int_{\mathbb{R}^3} |u|^2\right)^{\frac{12-q}{10}}\left(4\int_{\mathbb{R}^3} |u|^2 |\nabla u|^2\right)^{\frac{3(q-2)}{10}},\;\; \mbox{for}\;\; \mbox{any}\;\; q \in (2,12).
 \end{equation}
Inspired by \cite{J1997}, we consider the $L^2$-norm preserved transform
\begin{equation} \label{eq2-5}
s*u(x)=e^{\frac{3}{2}s} u(e^sx),
\end{equation}
it is easy to see that
$$I_\mu(s*u)=\frac{\mu}{\theta} e^{\theta(1+\gamma_\theta)s}  \int_{\mathbb{R}^3} |\nabla u|^{\theta} + \frac{1}{2}e^{2s} \int_{\mathbb{R}^3} |\nabla u|^2 + e^{5s}\int_{\mathbb{R}^3} |u|^2 |\nabla u|^2 -e^{-3s}\int_{\mathbb{R}^3}  H\left(e^{\frac{3}{2}s}u\right) ,$$
\begin{align*}
\frac{d}{ds}I_\mu(s*u)=&\;\mu(1+\gamma_\theta) e^{\theta(1+\gamma_\theta)s}  \int_{\mathbb{R}^3} |\nabla u|^{\theta} + e^{2s} \int_{\mathbb{R}^3} |\nabla u|^2 + 5e^{5s}\int_{\mathbb{R}^3} |u|^2 |\nabla u|^2 \\
&+3e^{-3s}\int_{\mathbb{R}^3}  H\left(e^{\frac{3}{2}s}u\right)-\frac{3}{2} e^{-\frac{3}{2}s}\int_{\mathbb{R}^3}  h\left(e^{\frac{3}{2}s}u\right)u.
\end{align*}
As a result,
$$P_\mu(u)=\frac{d}{ds}\bigg|_{s=0}I_\mu(s*u).$$

Next, we need to establish the following important lemmas.
\begin{lemma}\label{lem2.1}
 If h satisfies $(h_1)$, $(h_2)$, $(h_4)$ and $(h_5')$, then
 $$h(t)t>\frac{16}{3}H(t)>0, \;  for\; any \; t \neq 0.$$
\end{lemma}
\begin{proof}
The proof is inspired by \cite[Lemma 2.3]{JL2020}. We split the proof into several claims.

\textbf{Claim 1:}
\;$H(t)>0$, for any $t \neq 0.$

Indeed, if $H(t_0)\leq 0$ for some $t_0 \neq 0$, by $(h_1)$, $(h_2)$ and $(h_4)$, there exists $\tau \neq 0$ such that
$$H(\tau)\leq 0,\;\,\frac{H(\tau)}{|\tau|^{\frac{16}{3}}} =\min_{t \in \mathbb{R}} \frac{H(t)}{|t|^{\frac{16}{3}}}$$
and
$$\left[\frac{H(t)}{|t|^{\frac{16}{3}}}\right]_{t=\tau}' =\frac{h(\tau)\tau-\frac{16}{3}H(\tau)} {|\tau|^{\frac{19}{3}}} \mbox{sgn}(\tau) =0.$$
Set
\begin{align*}
\begin{split}
g(t):=\left \{
\begin{array}{ll}
\frac{h(t)t-2H(t)}{|t|^{\frac{16}{3}}},  &t \neq 0,\\
0,     &t=0.
\end{array}
\right.
\end{split}
\end{align*}

By $(h_2)$ and $(h_5')$, we obtain $g$ is strictly decreasing on $(-\infty,0]$ and strictly increasing on $[0,+\infty)$. Then $h(t)t>2H(t)$ for any $t \neq 0$, we derive a contradiction:
$$0<{h(\tau)\tau-2H(\tau)}=\frac{10}{3}H(\tau) \leq 0.$$
The proof of Claim 1 is complete.

\textbf{Claim 2:}
 There exist a positive sequence $\{\tau^+_n\}$  and a negative sequence $\{\tau^-_n\}$  such that $|\tau^\pm_n| \rightarrow 0$ and $h(\tau^\pm_n) \tau^\pm_n >  \frac{16}{3}H(\tau^\pm_n )$  for each $n \geq 1$.

We first consider the positive case. By contradiction, we assume that there exists a constant $T_s>0$ small enough such that $h(t)t \leq \frac{16}{3}H(t)$, for any $t \in (0,T_s]$. Then
$$\left [\frac{H(t)}{|t|^{\frac{16}{3}}}\right]'\leq 0.$$
Using Claim 1, we have
$$ \frac{H(t)}{|t|^{\frac{16}{3}}}\geq  \frac{H(T_s)}{|T_s|^{\frac{16}{3}}}>0.$$
By $(h_2)$, we obtain a contradiction. The negative case is similar and so we obtain Claim 2.

\textbf{Claim 3:}
 There exist a positive sequence $\{\sigma^+_n\}$  and a negative sequence $\{\sigma^-_n\}$  such that $|\sigma^\pm_n| \rightarrow +\infty$ and $h(\sigma^\pm_n) \sigma^\pm_n >  \frac{16}{3}H(\sigma^\pm_n )$  for each $n \geq 1$.

Claim 3 is similar to Claim 2.

\textbf{Claim 4:}
$h(t)t \geq \frac{16}{3}H(t)$, for any $t \neq 0.$

Let us assume that $h(t_0)t_0 < \frac{16}{3}H(t_0)$ for some $t_0\neq 0$. Without loss of generality, we can take $t_0 > 0$. Claim 2 and Claim 3 imply that there exist $\xi_1 \in (0,t_0)$ and
 $\xi_2 \in (t_0, +\infty)$ such that
\begin{equation} \label{eq2-6}
h(t)t < \frac{16}{3}H(t), \;\mbox{for any }t \in (\xi_1,\xi_2),
\end{equation}
\begin{equation} \label{eq2-7}
h(t)t = \frac{16}{3}H(t),\;\mbox{when} \ t=\xi_1,\xi_2.
\end{equation}
In view of \eqref{eq2-6}, we have
\begin{equation} \label{eq2-8}
\frac{H(\xi_1)}{|\xi_1|^{\frac{16}{3}}}>  \frac{H(\xi_2)}{|\xi_2|^{\frac{16}{3}}}.
\end{equation}
On the other hand, by \eqref{eq2-7} and $(h_5')$, it is clear that
\begin{equation} \label{eq2-9}
\frac{H(\xi_1)}{|\xi_1|^{\frac{16}{3}}} =\frac{3}{10}\frac{\tilde{H}(\xi_1)}{|\xi_1|^{\frac{16}{3}}} <                         \frac{3}{10}\frac{\tilde{H}(\xi_2)}{|\xi_2|^{\frac{16}{3}}}  = \frac{H(\xi_2)}{|\xi_2|^{\frac{16}{3}}}.
\end{equation}
So we obtain Claim 4.

It follows from Claim 4 that the function $\frac{H(t)}{|t|^{\frac{16}{3}}}$ is nonincreasing on $(-\infty,0)$ and nondecreasing on $(0,+\infty)$. For any  $t>0$, we get
$\frac{h(t)}{|t|^{\frac{13}{3}}}=\frac{\tilde{H}(t)}{|t|^{\frac{16}{3}}}+2\frac{H(t)}{|t|^{\frac{16}{3}}}$. By $(h_5')$, we conclude that the function $\frac{h(t)}{|t|^{\frac{13}{3}}}$ is strictly increasing on $(0,+\infty)$. For any $t\neq0$,
it is easy to see that
\begin{align*}
\frac{16}{3}H(t)&=\frac{16}{3}\int^t_0 h(s)ds \\
&<\frac{16}{3} \int^t_0 \frac{h(t)}{|t|^{\frac{13}{3}}} |s|^{\frac{13}{3}} ds \\
&=h(t)t,
\end{align*}
which completes the proof of Lemma \ref{lem2.1}.

\end{proof}

\begin{remark}
In fact, $h(t)=|t|^{p-2}t$ does not need this technical lemma, where $p \in (\frac{16}{3},6)$. While, for the general $h$, a lot of difficulties arise. The above argument is crucially used to establish  the coerciveness of $I_\mu$, which plays a significant role in the boundedness of the Palais-Smale sequence.
\end{remark}

We recall here a radial lemma in \cite{Willem}.
\begin{lemma}\label{lem02-1}
      If $N \geq 2$, then there exists $c(N)>0$ such that, for every $u \in H^1_r(\mathbb{R}^N)$,
      $$|u(x)| \leq c(N)\|u\|^\frac{1}{2}_2\|\nabla u\|^\frac{1}{2}_2|x|^\frac{1-N}{2} \;\,a.e. \;\,in \;\, \mathbb{R}^N.$$
\end{lemma}

\begin{lemma}\label{lem02-2}
    Assume that $h$ is a continuous function satisfying $(h_2)$ and $(h_3)$. If $u_n \rightharpoonup u$ in $H^1_r(\mathbb{R}^3)$ and $u_n \rightarrow u$ a.e. in $\mathbb{R}^3$  as $n \rightarrow \infty$, then we have  \\
(1) $\int_{\mathbb{R}^3} h(u_n)u_n \rightarrow \int_{\mathbb{R}^3} h(u)u$;  \\
(2) $\int_{\mathbb{R}^3} H(u_n) \rightarrow \int_{\mathbb{R}^3} H(u)$;  \\
(3) $\int_{\mathbb{R}^3} \tilde{H}(u_n) \rightarrow \int_{\mathbb{R}^3} \tilde{H}(u)$.
\end{lemma}

\begin{proof}
(1) By means of $(h_2)$ and Lemma \ref{lem02-1}, for any $\varepsilon >0$ there exists $R>0$ such that
$$\int_{\mathbb{R}^3 \backslash B_R(0)} (|h(u_n)u_n| + |h(u)u|)dx \leq \varepsilon \int_{\mathbb{R}^3 \backslash B_R(0)} \left(|u_n|^{\frac{16}{3}}+ |u|^{\frac{16}{3}}\right)dx \leq C\varepsilon.$$
Therefore, to prove the result, it suffices to show that
$$\int_{B_R(0)} h(u_n)u_n dx \rightarrow \int_{B_R(0)} h(u)u dx.$$
Taking into account $(h_2)$ and $(h_3)$, for $\varepsilon >0$ there exists $C_\varepsilon >0$ such that
\begin{equation}\label{eq02-9}
    |h(t)t| \leq C_\varepsilon +\varepsilon |t|^{6} \;\, \mbox{for} \;\, t \in \mathbb{R}.
\end{equation}
Denote $\delta:=\frac{\varepsilon}{C_\varepsilon}$. In view of the Egorov theorem, there exists $\Omega \subset B_R(0)$ with meas$(\Omega) <\delta$ such that $h(u_n)u_n \rightarrow h(u)u$ uniformly for $x \in B_R(0) \backslash \Omega$. Then
\begin{equation*}
   \int_{B_R(0) \backslash \Omega} |h(u_n)u_n - h(u)u| dx \rightarrow 0,
\end{equation*}
as $n \rightarrow \infty$. We see from \eqref{eq02-9} and the Sobolev inequality that
\begin{equation*}
   \int_{\Omega} |h(u_n)u_n - h(u)u| dx \leq \int_{\Omega} \left(2C_\varepsilon +\varepsilon |u_n|^{6} +\varepsilon |u|^{6}\right)dx \leq C\varepsilon.
\end{equation*}
Taking upper limit as $n \rightarrow \infty$ in
\begin{equation*}
   \int_{B_R(0)} |h(u_n)u_n - h(u)u| dx= \int_{B_R(0) \backslash \Omega} |h(u_n)u_n - h(u)u| dx+\int_{\Omega} |h(u_n)u_n - h(u)u| dx,
\end{equation*}
we have
\begin{equation*}
  \limsup_{n \rightarrow \infty} \int_{B_R(0)} |h(u_n)u_n - h(u)u| dx \leq C\varepsilon,
\end{equation*}
which indicates
$$\int_{B_R(0)} h(u_n)u_n dx \rightarrow \int_{B_R(0)} h(u)u dx.$$
The proof is finished.

(2) and (3) are similar to (1), we omit their proofs here.
\end{proof}

\begin{lemma} \label{lem2.2}
Let $(h_1)-(h_3)$ hold. Then there exists a constant $\zeta >0$ such that when $u \in \mathcal{S}(a)$ with $\int_{\mathbb{R}^3}|\nabla u|^2 + \int_{\mathbb{R}^3} |u|^2 |\nabla u|^2 <\zeta $, we have
\begin{equation}\label{eq2-10}
\frac{1}{4}\left(\int_{\mathbb{R}^3} |\nabla u|^2 +\int_{\mathbb{R}^3} |u|^2 |\nabla u|^2\right)\leq I_\mu(u) \leq   \frac{\mu}{\theta} \int_{\mathbb{R}^3} |\nabla u|^{\theta} +\int_{\mathbb{R}^3} |\nabla u|^2 +\int_{\mathbb{R}^3} |u|^2 |\nabla u|^2,
\end{equation}
\begin{equation}\label{eq2-10-1}
P_\mu (u)\geq \frac{1}{2} \left(\int_{\mathbb{R}^3} |\nabla u|^2 +\int_{\mathbb{R}^3} |u|^2 |\nabla u|^2\right).
\end{equation}
\end{lemma}

\begin{proof}
For any $u\in \mathcal{S}(a)$, taking into account \eqref{eq2-3}, \eqref{GNI} and the Sobolev embedding theorem we can deduce that
\begin{align*}
\left|\int_{\mathbb{R}^3} H(u)\right| & \leq \varepsilon \int_{\mathbb{R}^3} |u|^{\frac{16}{3}} +C_\varepsilon \int_{\mathbb{R}^3} |u|^{6}  \\
&\leq C\left[\varepsilon \int_{\mathbb{R}^3} |u|^2 |\nabla u|^2 +C_\varepsilon \left(\int_{\mathbb{R}^3}|\nabla u|^2\right)^3\right]\\
&\leq C\left[\varepsilon  +C_\varepsilon \left(\int_{\mathbb{R}^3}|\nabla u|^2\right)^{2}\right]
\left(\int_{\mathbb{R}^3} |u|^2 |\nabla u|^2 + \int_{\mathbb{R}^3}|\nabla u|^2 \right).
\end{align*}
Let $\varepsilon =\frac{1}{8C} $, there exists a constant $\zeta' >\int_{\mathbb{R}^3}|\nabla u|^2 + \int_{\mathbb{R}^3} |u|^2 |\nabla u|^2$ such that
\begin{align*}
\left|\int_{\mathbb{R}^3} H(u)\right| \leq \frac{1}{4} \left(\int_{\mathbb{R}^3}|\nabla u|^2 + \int_{\mathbb{R}^3} |u|^2 |\nabla u|^2\right).
\end{align*}
Therefore,
\begin{align*}
\frac{1}{4}\left(\int_{\mathbb{R}^3} |\nabla u|^2 +\int_{\mathbb{R}^3} |u|^2 |\nabla u|^2\right)\leq I_\mu(u) \leq   \frac{\mu}{\theta} \int_{\mathbb{R}^3} |\nabla u|^{\theta} +\int_{\mathbb{R}^3} |\nabla u|^2 +\int_{\mathbb{R}^3} |u|^2 |\nabla u|^2.
\end{align*}

Similarly, there exists a constant $\zeta'' >\int_{\mathbb{R}^3}|\nabla u|^2 + \int_{\mathbb{R}^3} |u|^2 |\nabla u|^2$ such that
\begin{align*}
\left|\frac{3}{2}\int_{\mathbb{R}^3} \tilde{H}(u)\right| \leq \frac{1}{4} \left(\int_{\mathbb{R}^3}|\nabla u|^2 + \int_{\mathbb{R}^3} |u|^2 |\nabla u|^2\right).
\end{align*}
Moreover,
\begin{align*}
P_\mu (u)&=\mu(1+\gamma_\theta) \int_{\mathbb{R}^3} |\nabla u|^{\theta} +\int_{\mathbb{R}^3} |\nabla u|^2 +5\int_{\mathbb{R}^3} |u|^2 |\nabla u|^2 -\frac{3}{2} \int_{\mathbb{R}^3} \tilde{H}(u) \\
&\geq \frac{1}{2} \left(\int_{\mathbb{R}^3} |\nabla u|^2 +\int_{\mathbb{R}^3} |u|^2 |\nabla u|^2\right).
\end{align*}

In the end, we may set $\zeta=\min\{\zeta',\zeta''\}$. Therefore, the conclusion is valid.

\end{proof}

\begin{remark}
Lemma \ref{lem2.2} gives us the estimation of upper and lower bounds of $I_\mu$, which has a crucial impact on the discussion of the geometry of $I_\mu$. Moreover, it is also used to prove the positivity of critical value by which we can obtain the existence of ground state normalized
solutions and infinitely many normalized solutions.

\end{remark}

Now, we introduce some properties about $I_\mu(s*u)$.
\begin{lemma} \label{lem2.3}
Suppose that $(h_1)-(h_4)$ and $(h_5')$ hold. Then for any $u \in \Upsilon \backslash \{0\}$, we have\\
(1) $I_\mu(s*u) \rightarrow 0^+$ as $s\rightarrow -\infty;$\\
(2) $I_\mu(s*u) \rightarrow -\infty $ as $s\rightarrow +\infty.$
\end{lemma}

\begin{proof}
(1) There holds $\int_{\mathbb{R}^3} |\nabla (s*u)|^2 + \int_{\mathbb{R}^3} |s*u|^2 |\nabla (s*u)|^2 <\zeta $ for any $s$ satisfying $s<0$ and $|s|$ large enough. It follows from Lemma \ref{lem2.2} that
\begin{align*}
0\leq I_\mu(s*u) \leq \frac{\mu}{\theta} e^{\theta(1+\gamma_\theta)s}  \int_{\mathbb{R}^3} |\nabla u|^{\theta} + e^{2s} \int_{\mathbb{R}^3} |\nabla u|^2 + e^{5s}\int_{\mathbb{R}^3} |u|^2 |\nabla u|^2 \rightarrow 0^+,
\end{align*}
 as $s\rightarrow -\infty$.

(2) From $(h_4)$, Lemma \ref{lem2.1} and the Fatou's lemma, we conclude that
\begin{align*}
\liminf_{s \rightarrow +\infty} e^{-8s}\int_{\mathbb{R}^3}H\left(e^{\frac{3}{2}s}u\right)&= \liminf_{s \rightarrow +\infty} \int_{\mathbb{R}^3}\frac{H\left(e^{\frac{3}{2}s}u\right)}{|e^{\frac{3}{2}s}u|^{\frac{16}{3}}}|u|^{\frac{16}{3}}    \\
&\geq  \int_{\mathbb{R}^3} \liminf_{s \rightarrow +\infty} \frac{H\left(e^{\frac{3}{2}s}u\right)}{|e^{\frac{3}{2}s}u|^{\frac{16}{3}}}|u|^{\frac{16}{3}}= +\infty.
\end{align*}
Then
\begin{align*}
&\ I_\mu(s*u)   \\
=&\ \frac{\mu}{\theta} e^{\theta(1+\gamma_\theta)s}  \int_{\mathbb{R}^3} |\nabla u|^{\theta} + \frac{1}{2}e^{2s} \int_{\mathbb{R}^3} |\nabla u|^2 + e^{5s}\int_{\mathbb{R}^3} |u|^2 |\nabla u|^2 -e^{-3s}\int_{\mathbb{R}^3}H\left(e^{\frac{3}{2}s}u\right) \\
 \leq&\ e^{5s}\bigg(\frac{\mu}{\theta} e^{-[5-\theta(1+\gamma_\theta)]s}  \int_{\mathbb{R}^3} |\nabla u|^{\theta} + \frac{1}{2}e^{-3s} \int_{\mathbb{R}^3} |\nabla u|^2 + \int_{\mathbb{R}^3} |u|^2 |\nabla u|^2   \\
 & -e^{-8s}\int_{\mathbb{R}^3}H\left(e^{\frac{3}{2}s}u\right) \bigg) \\
 \rightarrow& -\infty,
\end{align*}
as $s\rightarrow +\infty$.
\end{proof}

We define the Pohozaev manifold $\mathcal{P}_\mu(a):=\{u\in \mathcal{S}(a)\ |\ P_\mu(u)=0 \}$ and obtain the following properties about the Pohozaev manifold.

\begin{lemma}\label{lem2.4}
If h satisfies $(h_1)-(h_4)$ and $(h_5')$, then for any $u \in \Upsilon \backslash \{0\}$, we have  \\
(1) there exists a unique number $s(u) \in \mathbb{R}$ such that $P_\mu(s(u)*u)=0$, $I_\mu(s(u)*u)=\max_{s\in \mathbb{R}} I_\mu(s*u) >0$;  \\
(2) the mapping $u\mapsto s(u)$ is well defined and continuous on $ \Upsilon \backslash \{0\}$;  \\
(3) $s(u(\cdot+y)) = s(u)$ for any $y \in \mathbb{R}^3$, $s(-u) = s(u)$.
\end{lemma}

\begin{proof}
(1) It is easy to see that
$$ \frac{d}{ds}I_\mu(s*u)=P_\mu(s*u).$$
By using of Lemma \ref{lem2.3}, there exists some $s(u) \in \mathbb{R}$ such that $I_\mu(s*u)$ achieves the global maximum at $s(u)$ and $I_\mu(s(u)*u)>0$. Then, $ P_\mu(s(u)*u)=0.$

In what follows, we prove the uniqueness of $s(u)$. Suppose that $s_1(u)<s_2(u)$ such that
$$P_\mu(s_1(u)*u)=0=P_\mu(s_2(u)*u).$$
In view of the fact that
\begin{align*}
&\ P_\mu(s*u) \\
 =&\;\mu(1+\gamma_\theta) e^{\theta(1+\gamma_\theta)s}  \int_{\mathbb{R}^3} |\nabla u|^{\theta} + e^{2s} \int_{\mathbb{R}^3} |\nabla u|^2 + 5e^{5s}\int_{\mathbb{R}^3} |u|^2 |\nabla u|^2    \\
 &-\frac{3}{2}e^{5s}\int_{\mathbb{R}^3} \frac{\tilde{H}(e^{\frac{3}{2}s}u)}{|e^{\frac{3}{2}s}u|^{\frac{16}{3}}}|u|^{\frac{16}{3}} \\
=&\;e^{5s} \bigg[\mu(1+\gamma_\theta) e^{-[5-\theta(1+\gamma_\theta)]s}  \int_{\mathbb{R}^3} |\nabla u|^{\theta} + e^{-3s} \int_{\mathbb{R}^3} |\nabla u|^2 + 5\int_{\mathbb{R}^3} |u|^2 |\nabla u|^2    \\
&-\frac{3}{2}\int_{\mathbb{R}^3} \frac{\tilde{H}(e^{\frac{3}{2}s}u)}{|e^{\frac{3}{2}s}u|^{\frac{16}{3}}}|u|^{\frac{16}{3}} \bigg]
\end{align*}
and the mapping $s \mapsto \frac{\tilde{H}(e^{\frac{3}{2}s}u)}{|e^{\frac{3}{2}s}u|^{\frac{16}{3}}}$ is strictly increasing for $u \in \Upsilon\backslash \{0\}$ by $(h_5')$, we get $s_1(u)=s_2(u)$.

(2) By (1), it is easy to see that the mapping $u\mapsto s(u)$ is well defined on $ \Upsilon \backslash \{0\}.$
Let $\{u_n\} \subset \Upsilon \backslash \{0\}$ be any sequence satisfying $ u_n \rightarrow u \in \Upsilon \backslash \{0\}$ in $\Upsilon $.

\textbf{Claim:}
 $\{s(u_n)\}$ is bounded in $\mathbb{R}.$

Up to a subsequence, if $s(u_n) \rightarrow +\infty$ as $n \rightarrow \infty$, then
\begin{align*}
0\leq&\ e^{-5s(u_n)} I_\mu(s(u_n)*u_n) \\
\leq&\ \frac{\mu}{\theta}e^{-[5-\theta(1+\gamma_\theta)]s(u_n)}  \int_{\mathbb{R}^3} |\nabla u_n|^{\theta} + \frac{1}{2}e^{-3s(u_n)} \int_{\mathbb{R}^3} |\nabla u_n|^2 + \int_{\mathbb{R}^3} |u_n|^2 |\nabla u_n|^2     \\
&-e^{-8s(u_n)}\int_{\mathbb{R}^3}H\left(e^{\frac{3}{2}s(u_n)}u_n\right)  \\
\rightarrow& -\infty,
\end{align*}
as $n \rightarrow \infty$, a contradiction, which implies that $\{s(u_n)\}$ is bounded from above.

Up to a subsequence, if $s(u_n) \rightarrow -\infty$ as $n \rightarrow \infty$, then it follows from $s(u)*u_n \rightarrow s(u)*u$ in $\Upsilon$ that
\begin{align*}
0&<I_\mu(s(u)*u) \\
&=\lim_{n\rightarrow \infty}I_\mu(s(u)*u_n) \\
&\leq \liminf_{n\rightarrow \infty}I_\mu(s(u_n)*u_n) \\
&\leq \liminf_{n\rightarrow \infty}\left( \frac{1}{\theta} \int_{\mathbb{R}^3} |\nabla (s(u_n)*u_n)|^{\theta} + \frac{1}{2}\int_{\mathbb{R}^3} |\nabla (s(u_n)*u_n)|^2 + \int_{\mathbb{R}^3} |s(u_n)*u_n|^2 |\nabla (s(u_n)*u_n)|^2 \right) \\
&\leq \liminf_{n\rightarrow \infty} \left(\frac{1}{\theta}e^{\theta(1+\gamma_\theta)s(u_n)}  \int_{\mathbb{R}^3} |\nabla u_n|^{\theta} + \frac{1}{2}e^{2s(u_n)} \int_{\mathbb{R}^3} |\nabla u_n|^2 + e^{5s(u_n)}\int_{\mathbb{R}^3} |u_n|^2 |\nabla u_n|^2 \right) \\
&=0,
\end{align*}
a contradiction. This shows $\{s(u_n)\}$ is bounded in $\mathbb{R}$.

Without loss of generality, we may assume that $s(u_n) \rightarrow s$ as $n \rightarrow \infty$, which together with the fact $u_n \rightarrow u$ in $\Upsilon$ yields that $s(u_n)*u_n \rightarrow s*u$ in $\Upsilon$. Hence, we get that
\begin{align*}
&\ P_\mu(s*u) \\
=&\ \mu(1+\gamma_\theta) \int_{\mathbb{R}^3} |\nabla (s*u)|^{\theta} + \int_{\mathbb{R}^3} |\nabla (s*u)|^2 + 5\int_{\mathbb{R}^3} |s*u|^2 |\nabla (s*u)|^2 -\frac{3}{2}\int_{\mathbb{R}^3}\tilde{H}(s*u) \\
=&\lim_{n\rightarrow \infty} \bigg[ \mu(1+\gamma_\theta) \int_{\mathbb{R}^3} |\nabla (s(u_n)*u_n)|^{\theta} + \int_{\mathbb{R}^3} |\nabla (s(u_n)*u_n)|^2  \\
&+ 5\int_{\mathbb{R}^3} |s(u_n)*u_n|^2 |\nabla (s(u_n)*u_n)|^2 -\frac{3}{2}\int_{\mathbb{R}^3}\tilde{H}(s(u_n)*u_n) \bigg] \\
= &\lim_{n\rightarrow \infty} P_\mu(s(u_n)*u_n)=0.
\end{align*}
Thus, we can conclude that $s(u)=s$ and $s(u_n) \rightarrow s(u)$ as $n \rightarrow \infty$.

(3) For any $y \in \mathbb{R}^3$, by changing variables in the integrals, we have
$$P_\mu(s(u)*u(\cdot + y)) = P_\mu(s(u)*u)=0,$$
thus $s(u(\cdot+ y)) = s(u)$ via (1). It is clear that
$$P_\mu(s(u)*(-u)) = P_\mu(-s(u)*u)=P_\mu(s(u)*u)=0,$$
hence $s(-u)=s(u).$
\end{proof}

\begin{remark}
In special case $h(t)=|t|^{p-2}t$, it is easy to analyze the properties of $P_\mu$, while it is much more complicated for general $h$, and we need $(h_5')$ to get uniqueness of $s(u)$. As for the continuity, the implicit function theorem is a good tool. However, it is not possible to solve the problem for general $h$, and we need to consider new ways to solve the problem.
\end{remark}

Set
$$A_\mu^k:=\left\{u \in \mathcal{S}(a) \;\bigg|\;\mu\int_{\mathbb{R}^3} |\nabla u|^{\theta}+\int_{\mathbb{R}^3} |\nabla u|^2+\int_{\mathbb{R}^3} |u|^2 |\nabla u|^2\leq k\right\},$$
$$\partial A_\mu^k:=\left\{u \in \mathcal{S}(a)\;\bigg|\; \mu\int_{\mathbb{R}^3} |\nabla u|^{\theta}+\int_{\mathbb{R}^3} |\nabla u|^2+\int_{\mathbb{R}^3} |u|^2 |\nabla u|^2=k\right\}.$$

\begin{lemma}\label{lem2.5}
If h satisfies $(h_1)-(h_4)$ and $(h_5')$,  for $\mu \in (0,1]$, we have  \\
(1) $\inf_{u \in \mathcal{P}_\mu(a)} I_\mu(u) \geq \frac{\zeta}{8}>0$;  \\
(2) there exists a small $k_0>0$ independent of $\mu$ such that
$$\inf_{u \in \partial A_\mu^k}  P_\mu(u) >0 \;for \;any\; 0<k\leq k_0, $$
$$\sup_{u\in A_\mu^{k_0}} I_\mu(u) <\inf_{u \in \mathcal{P}_\mu(a)} I_\mu(u);$$  \\
(3) let $\{u_n\} \subset \mathcal{P}_\mu(a)$, if  $\sup_{n \in \mathbb{N}^+} I_\mu(u_n)<+\infty$, then there exists a constant $C>0$ such that
$$ \mu \int_{\mathbb{R}^3} |\nabla u_n|^{\theta}\leq C,\; \int_{\mathbb{R}^3} |\nabla u_n|^2 \leq C,\;  \int_{\mathbb{R}^3} |u_n|^2 |\nabla u_n|^2 \leq C, \;for\;any \;n \in \mathbb{N}^+.$$
\end{lemma}

\begin{proof}
(1) For $u \in \mathcal{P}_\mu(a)$, we take $s$ such that
$$ e^{-2s}\int_{\mathbb{R}^3}|\nabla u|^2 + e^{-5s}\int_{\mathbb{R}^3} |u|^2 |\nabla u|^2=\frac{\zeta}{2}.$$
Set $\tilde{u}:=(-s)*u$, it is clear that
$$\int_{\mathbb{R}^3}|\nabla \tilde{u}|^2 + \int_{\mathbb{R}^3} |\tilde{u}|^2 |\nabla \tilde{u}|^2=\frac{\zeta}{2}<\zeta,$$
 by using of Lemma \ref{lem2.2}, we can deduce that
\begin{align*}
&\; I_\mu(u) =I_\mu(0*u) \geq I_\mu(-s*u)  \\
\geq&\; \frac{1}{4} \left(\int_{\mathbb{R}^3} |\nabla \tilde{u}|^2 +\int_{\mathbb{R}^3} |\tilde{u}|^2 |\nabla \tilde{u}|^2\right) \\
\geq&\; \frac{\zeta}{8} >0.
\end{align*}

(2) We can take a small $k'_0>0$ independent of $\mu$ such that
$$\sup_{u\in A_\mu^{k'_0}} I_\mu(u) \leq k'_0 < \inf_{u \in \mathcal{P}_\mu(a)} I_\mu(u),$$
and thus, there exists a $k_0<k_0'$ such that for any $u \in \partial A_\mu^k$ with $0<k\leq k_0$, we get
\begin{align*}
&\ P_\mu (u)  \\
=&\ \mu(1+\gamma_\theta) \int_{\mathbb{R}^3} |\nabla u|^{\theta} +\int_{\mathbb{R}^3} |\nabla u|^2 +5\int_{\mathbb{R}^3} |u|^2 |\nabla u|^2 -\frac{3}{2} \int_{\mathbb{R}^3} \tilde{H}(u) \\
\geq&\ \mu(1+\gamma_\theta) \int_{\mathbb{R}^3} |\nabla u|^{\theta} +\int_{\mathbb{R}^3} |\nabla u|^2 +5\int_{\mathbb{R}^3} |u|^2 |\nabla u|^2  \\
&\,- C\left[ \varepsilon+ C_\varepsilon \left(\int_{\mathbb{R}^3}|\nabla u|^2\right)^2\right]
\left(\int_{\mathbb{R}^3} |u|^2 |\nabla u|^2 + \int_{\mathbb{R}^3} |\nabla u|^2 \right) \\
\geq&\;  \left\{1- C\left[ \varepsilon + C_\varepsilon \left(\int_{\mathbb{R}^3}|\nabla u|^2\right)^2\right] \right\}
\left(\mu \int_{\mathbb{R}^3} |\nabla u|^{\theta} + \int_{\mathbb{R}^3} |\nabla u|^2 +\int_{\mathbb{R}^3} |u|^2 |\nabla u|^2\right) \\
\geq&\; \frac{1}{2}k>0
\end{align*}
and
$$\sup_{u\in A_\mu^{k_0}} I_\mu(u) \leq\sup_{u\in A_\mu^{k'_0}} I_\mu(u)< \inf_{u \in \mathcal{P}_\mu(a)} I_\mu(u).$$

(3) From $u_n \in \mathcal{P}_\mu(a)$ and Lemma \ref{lem2.1}, we can deduce that
\begin{align*}
&\ \mu(1+\gamma_\theta) \int_{\mathbb{R}^3} |\nabla u_n|^{\theta} +\int_{\mathbb{R}^3} |\nabla u_n|^2 +5\int_{\mathbb{R}^3} |u_n|^2 |\nabla u_n|^2 \\
=& \ \frac{3}{2} \int_{\mathbb{R}^3} \left[h(u_n)u_n-2H(u_n)\right] \\
>&\ 5\int_{\mathbb{R}^3} H(u_n).
\end{align*}
Obviously,
\begin{align*}
I_\mu(u_n)&=\frac{\mu}{\theta} \int_{\mathbb{R}^3} |\nabla u_n|^{\theta} +\frac{1}{2}\int_{\mathbb{R}^3} |\nabla u_n|^2 +\int_{\mathbb{R}^3} |u_n|^2 |\nabla u_n|^2 -\int_{\mathbb{R}^3} H(u_n) \\
&>\frac{5-\theta(1+\gamma_\theta)}{5\theta} \mu \int_{\mathbb{R}^3} |\nabla u_n|^{\theta} +\frac{3}{10}\int_{\mathbb{R}^3} |\nabla u_n|^2,
\end{align*}
which implies
$$ \mu \int_{\mathbb{R}^3} |\nabla u_n|^{\theta}\leq C\;\;\mbox{and} \; \int_{\mathbb{R}^3} |\nabla u_n|^2 \leq C.$$
Next, we prove that
$$\limsup_{n \rightarrow \infty}\int_{\mathbb{R}^3} |u_n|^2 |\nabla u_n|^2 \leq C.$$
Assume that up to a subsequence, $\int_{\mathbb{R}^3} |u_n|^2 |\nabla u_n|^2 \rightarrow +\infty$ as $n \rightarrow \infty$. Take $s_n$ such that
$$e^{-\theta(1+\gamma_\theta)s_n} \left(\mu \int_{\mathbb{R}^3} |\nabla u_n|^{\theta}\right) +e^{-2s_n} \int_{\mathbb{R}^3} |\nabla u_n|^2 + e^{-5s_n}\int_{\mathbb{R}^3} |u_n|^2 |\nabla u_n|^2=1.$$
Set $v_n:=(-s_n)*u_n$, it is clear that
$$ \mu \int_{\mathbb{R}^3} |\nabla v_n|^{\theta} + \int_{\mathbb{R}^3} |\nabla v_n|^2 +\int_{\mathbb{R}^3} |v_n|^2 |\nabla v_n|^2=1,$$
  $$\int_{\mathbb{R}^3} |v_n|^2= \int_{\mathbb{R}^3} |u_n|^2=a^2\;\; \mbox{and}\;\; s(v_n)=s_n \rightarrow +\infty.$$
Let
$$\rho :=\limsup_{n \rightarrow \infty} \sup_{y \in \mathbb{R}^3} \int_{B_1(y)} |v_n|^2.$$

\textbf{Case 1:}
 If $\rho=0$, then $v_n \rightarrow 0$ in $L^{\frac{16}{3}}(\mathbb{R}^3)$ by \cite[Lemma 1.21]{Willem}. It follows from \eqref{eq2-4} that for any $\varepsilon >0$, we have $\int_{\mathbb{R}^3} H(v_n)\leq \varepsilon \int_{\mathbb{R}^3} |v_n|^{6} +C_\varepsilon\int_{\mathbb{R}^3} |v_n|^{\frac{16}{3}}$. It is easy to see that $C_\varepsilon \int_{\mathbb{R}^3} |v_n|^{\frac{16}{3}}<\varepsilon$ for any $n>0$ large enough. As a result,
\begin{align*}
    \int_{\mathbb{R}^3} H(v_n) &\leq \varepsilon \int_{\mathbb{R}^3} |v_n|^{6} +C_\varepsilon\int_{\mathbb{R}^3} |v_n|^{\frac{16}{3}} \\
    &\leq \varepsilon C \left(\int_{\mathbb{R}^3}|\nabla v_n|^2 \right)^3 +\varepsilon \\
    &\leq \left(C +1\right)\varepsilon,
\end{align*}
which means
\begin{equation}\label{eq02-12}
    \lim_{n \rightarrow \infty} e^{-3s}\int_{\mathbb{R}^3} H\Big(e^{\frac{3}{2}s}v_n\Big)=0 \;\; \mbox{for} \;\; \mbox{any}\;\; s>0.
\end{equation}
Since $P_\mu(s(v_n)*v_n)=P_\mu(u_n)=0$, we obtain that for $s>0$,
\begin{align*}
    C&\geq I_\mu(s(v_n)*v_n) \geq I_\mu(s*v_n) \\
      &= \frac{1}{\theta} e^{\theta(1+\gamma_\theta)s} \left(\mu \int_{\mathbb{R}^3} |\nabla v_n|^{\theta}\right) + \frac{1}{2}e^{2s} \int_{\mathbb{R}^3} |\nabla v_n|^2 + e^{5s}\int_{\mathbb{R}^3} |v_n|^2 |\nabla v_n|^2 -e^{-3s}\int_{\mathbb{R}^3}  H\left(e^{\frac{3}{2}s}v_n\right) \\
      &\geq \frac{1}{3}e^{2s} \left(\mu \int_{\mathbb{R}^3} |\nabla v_n|^{\theta} + \int_{\mathbb{R}^3} |\nabla v_n|^2 +\int_{\mathbb{R}^3} |v_n|^2 |\nabla v_n|^2\right) +o_n(1) \\
      &=\frac{1}{3}e^{2s} +o_n(1).
\end{align*}
Clearly, this leads to contradiction for $s>\max\left\{0,\frac{\ln(3C)}{2}\right\}$.

\textbf{Case 2:}
    If $\rho >0$, up to a subsequence, we may assume the existence of $\{y_n\} \subset \mathbb{R}^3$ such that
    $$\int_{B_1(y_n)} |v_n|^2> \frac{\rho}{2},$$
by changing variables in the integrals, we have
 $$\int_{B_1(0)} |v_n(\cdot +y_n)|^2> \frac{\rho}{2}.$$
Since $\{v_n(\cdot +y_n)\}$ is bounded in $H^1(\mathbb{R}^3)$, we suppose, up to a subsequence
\begin{align*}
 &v_n(\cdot +y_n) \rightharpoonup w     \;\;\;\,\;\;\; \;\;\;\;\;\;\;\;\;\;\;  \mbox{in} \; H^1(\mathbb{R}^3) , \\
 &v_n(\cdot +y_n) \rightarrow  w \neq 0       \;\;\;\;\;\;\;\;\;\;\;   \mbox{in} \;L^2_{loc}(\mathbb{R}^3),\\
 &v_n(\cdot +y_n) \rightarrow  w   \;\;\;\;\;\;\;\;\;\;\;\;\; \;\;\;\; \,             \mbox{a.e.} \; \mbox{in} \; \mathbb{R}^3.
\end{align*}
Set $z_n:=v_n(\cdot +y_n)$. Thus, Lemma \ref{lem2.4}(3) gives us that
$$s(z_n)=s(v_n(\cdot +y_n))=s(v_n) \rightarrow + \infty.$$
From $(h_4)$, Lemma \ref{lem2.1} and the Fatou's lemma, we get
$$\lim_{n \rightarrow +\infty} e^{-8s(z_n)}\int_{\mathbb{R}^3}H\left(e^{\frac{3}{2}s(z_n)}z_n\right)= +\infty.$$
As a result,
\begin{align*}
0\leq&\ e^{-5s(z_n)} I_\mu(s(z_n)*z_n) \\
\leq&\ \frac{1}{\theta}e^{-[5-\theta(1+\gamma_\theta)]s(z_n)}  \left(\mu \int_{\mathbb{R}^3} |\nabla z_n|^{\theta}\right) + \frac{1}{2}e^{-3s(z_n)} \int_{\mathbb{R}^3} |\nabla z_n|^2   \\
&\,+ \int_{\mathbb{R}^3} |z_n|^2 |\nabla z_n|^2 -e^{-8s(z_n)}\int_{\mathbb{R}^3}H\left(e^{\frac{3}{2}s(z_n)}z_n\right)  \\
\rightarrow&\, -\infty,
\end{align*}
as $n \rightarrow \infty$, a contradiction. Therefore,
$$\limsup_{n \rightarrow \infty}\int_{\mathbb{R}^3} |u_n|^2 |\nabla u_n|^2 \leq C.$$
So the conclusion holds.

\end{proof}

\begin{remark}
Lemma \ref{lem2.5}(3) implies that $I_\mu$ is coercive on $\mathcal{P}_\mu(a)$, and the proof is strongly dependent on Lemma \ref{lem2.1}. If $\tilde{h}(t)t>p\tilde{H}(t)$ with $p \in \left(\frac{16}{3},6\right)$, it is easy for us to etablish the coerciveness
of $I_\mu$ by the Pohozaev identity. However, the coerciveness becomes more cumbersome  for $p=\frac{16}{3}$ in assumption $(h_5)$, so we use a new approach to get it.
\end{remark}

\section{Palais-Smale sequence}\label{Palais-Smale}

In general, a Palais-Smale type compactness condition is needed. However, it seems impossible to prove that $I_\mu$ has such a compactness property. To overcome this difficulty, we consider an auxiliary functional  $\Phi_\mu:\mathbb{R}\times \Upsilon \mapsto \mathbb{R}$ defined in \cite{J1997} by
\begin{equation}\label{eq2-11}
\Phi_\mu(s,u):=I_\mu(s*u) .
\end{equation}

We study $\Phi_\mu$ on the radial space $\mathbb{R} \times \mathcal{S}_r(a)$ with
$$\mathcal{S}_r(a):=\mathcal{S}(a)\cap \Upsilon_{r}.$$
It is easy to see that $\Phi_\mu$ is of class $\mathcal{C}^1$. By the symmetric critical point principle \cite{P1979}, a Palais-Smale sequence for $\Phi_\mu|_{\mathbb{R} \times \mathcal{S}_r(a)}$ is also a Palais-Smale sequence for $\Phi_\mu|_{\mathbb{R} \times \mathcal{S}(a)}$. We define the closed sublevel set and minimax level
\begin{equation}\label{eq2-12}
I_\mu^d:=\{u\in \mathcal{S}(a)\;|\;I_\mu(u)\leq d\},
\end{equation}
\begin{equation}\label{eq2-13}
m_\mu(a):=\inf_{\gamma \in \Gamma_\mu} \sup_{t \in [0,1]} \Phi_\mu(\gamma(t)),
\end{equation}
where
$$ \Gamma_\mu:=\left\{\gamma=(\alpha,\beta) \in \mathcal{C}([0,1],\mathbb{R}\times \mathcal{S}_r(a))\;| \;\gamma(0) \in \{0\}\times A_\mu^{k_0},\gamma(1) \in \{0\}\times I_\mu^0 \right\}.$$

\begin{definition}\label{def3.1}
\cite[Definition 3.1]{G1993} Let B be a closed subset of X. We say that a class $\mathcal{F}$ of compact subsets of X is a homotopy stable family with boundary B provided  \\
(a) every set in $\mathcal{F}$ contains B;  \\
(b) for any set A in $\mathcal{F}$ and any $\eta \in \mathcal{C}([0,1]\times X,X)$ satisfying $\eta(t,x)=x$ for all $(t,x)$ in $(\{0\}\times X)\cup ([0,1]\times B) $ we have that $\eta(1,A) \in\mathcal{F}.$
\end{definition}

\begin{proposition}\label{prop3.2}
\cite[Theorem 5.2]{G1993} Let $\phi$ be a $\mathcal{C}^1$-functional on a complete connected $\mathcal{C}^1$-Finsler manifold X and consider a homotopy stable family $\mathcal{F}$ with an extended closed boundary B. Set $c=c(\phi,\mathcal{F})$ and let F be a closed subset of X satisfying
\begin{equation}\label{eq3.4}
A\cap F\backslash B\neq\emptyset \;\;\; for \;all \;A \in \mathcal{F}
\end{equation}
and
\begin{equation}\label{eq3.5}
\sup\phi(B)\leq c \leq\inf\phi(F).
\end{equation}
Then for any sequence of sets $\{A_n\} \subset \mathcal{F} $ satisfying $\lim_{n\rightarrow \infty} \sup_{A_n} \phi=c$, there exists a sequence $x_n \subset X \backslash B$ such that  \\
(1) $\lim_{n\rightarrow\infty}\phi(x_n)=c;$  \\
(2) $\lim_{n\rightarrow\infty}\|\phi'(x_n)\|=0;$  \\
(3) $\lim_{n\rightarrow\infty}{\rm dist}(x_n,F)=0;$  \\
(4) $\lim_{n\rightarrow\infty}{\rm dist}(x_n,A_n)=0.$
\end{proposition}

\begin{lemma}\label{lem-P}
Assume $h$ satisfies $(h_5)$ and $\mu \in(0,1]$, then $\mathcal{P}_\mu(a)$ is a $\mathcal{C}^1$-submanifold of codimension 1 in $\mathcal{S}(a)$.
\end{lemma}

\begin{proof}
We define $L(a):=\int_{\mathbb{R}^3}|u|^2-a^2$, clearly $L \in \mathcal{C}^1 (\Upsilon)$.

\textbf{Claim:}
 $d(P_\mu,L): \Upsilon \mapsto \mathbb{R}^2$ is surjective.

If $dP_\mu(u)$ and $dL(u)$ are linearly dependent, then there exists a $l \in \mathbb{R}$ such that
\begin{align*}
&\;2l \int_{\mathbb{R}^3}u\varphi  \\
=&\;\theta \mu (1+\gamma_\theta) \int_{\mathbb{R}^3} |\nabla u|^{\theta-2}\nabla u \cdot \nabla \varphi
+2\int_{\mathbb{R}^3} \nabla u \cdot \nabla \varphi
+10\int_{\mathbb{R}^3}\left(|u|^2 \nabla u \cdot \nabla \varphi +u\varphi|\nabla u|^2\right)  \\
&-\frac{3}{2}\int_{\mathbb{R}^3}\tilde{h}(u)\varphi 
\end{align*}
for any $\varphi \in \Upsilon$. Similar to the Pohozaev identity, let $\varphi=u$ and $\varphi=x\cdot \nabla u$, we obtain
\begin{align*}
&\theta\mu (1+\gamma_\theta)^2 \int_{\mathbb{R}^3} |\nabla u|^{\theta}  +2\int_{\mathbb{R}^3} |\nabla u|^2
+25\int_{\mathbb{R}^3}|u|^2|\nabla u|^2
-\frac{9}{4}\int_{\mathbb{R}^3}\Big[\tilde{h}(u)u-2\tilde{H}(u)\Big]=0.
\end{align*}
Due to $P_\mu (u)=0 $, we conclude
\begin{align*}
&[5-\theta (1+\gamma_\theta)](1+\gamma_\theta)\mu \int_{\mathbb{R}^3} |\nabla u|^{\theta}  +3\int_{\mathbb{R}^3} |\nabla u|^2
+\frac{9}{4}\int_{\mathbb{R}^3}\left[\tilde{h}(u)u-\frac{16}{3}\tilde{H}(u)\right]=0.
\end{align*}
By $(h_5)$, we have $u=0$. This contradicts with $u \in \mathcal{S}(a).$

\end{proof}

\begin{lemma}\label{lem3.3}
If h satisfies $(h_1)-(h_4)$ and $(h_5')$, then for $\mu \in (0,1]$, we have
 $$m_\mu(a)=\inf_{u \in \mathcal{P}_\mu(a)} I_\mu(u).$$
\end{lemma}

\begin{proof}
For any $\gamma=(\alpha,\beta) \in \Gamma_\mu$, we consider an auxiliary function
$$p(t):=P_\mu(\alpha(t)*\beta(t)).$$
By Lemma \ref{lem2.5}(2), we get that $p(0)=P_\mu(\beta(0))>0$. Lemma \ref{lem2.3} and Lemma \ref{lem2.4}(1) mean that $s(u)<0$ if and only if $P_\mu(u)<0$. Since $I_\mu(\beta(1)) \leq 0$, we have $s(\beta(1))<0$, which yields that $p(1)=P_\mu(\beta(1))<0$. Moreover, $p$ is continuous, thus there exists $\tau \in (0,1)$ such that $p(\tau)=0$, i.e., $\alpha(\tau)*\beta(\tau) \in \mathcal{P}_\mu(a)$.

As a result,
$$ \max_{t\in[0,1]}\Phi_\mu(\gamma(t)) \geq I_\mu(\alpha(\tau)*\beta(\tau)) \geq \inf_{u \in \mathcal{P}_\mu(a)} I_\mu(u).$$
Consequently,
$$m_\mu(a)\geq \inf_{u \in \mathcal{P}_\mu(a)} I_\mu(u).$$

If $u \in \mathcal{P}_\mu(a) \cap \Upsilon_r$, then there exist large $|\xi_1|$ and $|\xi_2|$ such that $\xi_1<0$, $\xi_2>0$ and
$(0,((1-t)\xi_1+t\xi_2)*u) \in \Gamma_\mu.$
By using of the symmetric decreasing rearrangement \cite{LL1997}, we deduce that
\begin{align*}
\inf_{u \in \mathcal{P}_\mu(a)} I_\mu(u) \geq \inf_{u \in \mathcal{P}_\mu(a) \cap \Upsilon_r} I_\mu(u) \geq \inf_{u \in \mathcal{P}_\mu(a) \cap \Upsilon_r} \max_{t \in [0,1]} I_\mu(((1-t)\xi_1+t\xi_2)*u) \geq m_\mu(a).
\end{align*}
Obviously, $m_\mu(a)=\inf_{u \in \mathcal{P}_\mu(a)} I_\mu(u).$

\end{proof}

\begin{remark}\label{rem3.5}
It is clear that $m_\mu(a)\geq \frac{\zeta}{8}>0$ and $m_\mu(a)>\sup_{u \in A^{k_0}_\mu}I_\mu(u)$.
\end{remark}

Next, we will prove the existence of the Palais-Smale sequence.

\begin{lemma} \label{lem3.4}
If h satisfies $(h_1)-(h_5)$, then for any fixed $\mu \in (0,1]$, there exists a sequence $\{u_n\} \subset \mathcal{S}_r(a)$ such that  \\
 (1) $I_\mu(u_n) \rightarrow m_\mu(a);$  \\
(2) $I_\mu|'_{\mathcal{S}(a)}(u_n) \rightarrow 0;$  \\
(3) $P_\mu(u_n) \rightarrow 0;$  \\
(4) $u^-_n \rightarrow 0 \; a.e. \; in \; \mathbb{R}^3.$
\end{lemma}

\begin{proof}
We will prove this lemma by applying Proposition \ref{prop3.2} to $\phi=\Phi_\mu,c=m_\mu(a)$ with
$$ X=\mathbb{R}\times \mathcal{S}_r(a),\;  B=\left(\{0\}\times A_\mu^{k_0}\right)\cup  \left(\{0\}\times I_\mu^0\right),$$
$$\mathcal{F}=\left\{A=\gamma([0,1])\;|\;\gamma \in \Gamma_\mu\right\},\;  F=\left\{(s,u)\;|\;\Phi_\mu(s,u) \geq m_\mu(a)\right\}.$$
Using Definition \ref{def3.1}, it is easy to check that $\mathcal{F}$ is a homotopy stable family of compact subsets of X with boundary B. According to the definition of $I_\mu^0$ and Remark \ref{rem3.5}, we get that assumption \eqref{eq3.4} is satisfied. Obviously, assumption \eqref{eq3.5} holds.

Therefore, we take a minimizing sequence $\{\gamma_n=(0,\beta_n)\} \subset \Gamma_\mu$ with $\beta_n \geq0$ a.e. in $\mathbb{R}^3$. It results that there exists a sequence $\{(s_n,w_n)\} \subset \mathbb{R}\times \mathcal{S}_r(a)$ for $\Phi_\mu|_{\mathbb{R}\times \mathcal{S}_r(a)}$ at level $m_\mu(a)$ such that
\begin{align}\label{eq3.6}
\begin{split}
     \left \{
     \begin{array}{ll}
        \Phi_\mu(s_n,w_n) \rightarrow m_\mu(a), \\
        \|\nabla_{\mathbb{R}\times \mathcal{S}_r(a)} \Phi_\mu(s_n,w_n) \| \rightarrow 0, \\
        |s_n|+{\rm dist}_\Upsilon(w_n,\beta_n([0,1]))  \rightarrow  0,
     \end{array}
     \right.
\end{split}
\end{align}
as $n \rightarrow \infty$.

 Setting $u_n=s_n*w_n$, the first condition and the third condition in \eqref{eq3.6} mean (1) and (4) hold.

 The second condition in \eqref{eq3.6} gives
\begin{align*}
\|I'_\mu|_{\mathcal{S}(a)}(u_n) \|_* &=\sup_{\psi \in T_{u_n}\mathcal{S}(a),\|\psi\|_\Upsilon\leq1} |I'_\mu(u_n)[\psi]| \\
&=\sup_{\psi \in T_{u_n}\mathcal{S}(a),\|\psi\|_\Upsilon\leq1} |I'_\mu(s_n*w_n)[s_n*(-s_n)*\psi]|  \\
&=\sup_{\psi \in T_{u_n}\mathcal{S}(a),\|\psi\|_\Upsilon\leq1} |\partial_u \Phi_\mu(s_n,w_n)[(-s_n)*\psi]|  \\
&\leq  \|\partial_u \Phi_\mu(s_n,w_n)\| \sup_{\psi \in T_{u_n}\mathcal{S}(a),\|\psi\|_\Upsilon\leq1} \|(-s_n)*\psi\|  \\
&\leq C\|\partial_u \Phi_\mu(s_n,w_n)\| \rightarrow 0
\end{align*}
and
$$ |P_\mu(u_n)|=\left|\frac{\partial \Phi_\mu(s_n,w_n)}{\partial s}\right| \rightarrow 0,$$
as $n \rightarrow \infty$. The proof of (2) and (3) is complete.
\end{proof}

\begin{remark}
    In Lemma \ref{lem-P}, the stronger condition $(h_5)$
is required. Therefore, we can't weaken $(h_5)$ to $(h'_
5)$
by utilizing Proposition \ref{prop3.2} to get a Palais-Smale sequence.
\end{remark}

Now, we can get the critical point of $I_\mu|_{\mathcal{S}(a)}$.

\begin{lemma}\label{lem3.5}
If h satisfies $(h_1)-(h_5)$, then for any fixed $\mu \in (0,1]$, there exist a nonnegative $u_\mu \in \Upsilon_r \backslash \{0\}$ and a $\lambda_\mu \in \mathbb{R}$ such that  \\
(1) $I'_\mu(u_\mu) +\lambda_\mu u_\mu=0;$  \\
(2) $P_\mu(u_\mu)=0;$  \\
(3) $I_\mu(u_\mu)=m_\mu(a);$  \\
(4) $0<\|u_\mu\|_2 \leq a;$ furthermore, if $\lambda_\mu \neq 0$, then $\|u_\mu\|_2 = a.$
\end{lemma}

\begin{proof}
(1) Let $u_n$ be a sequence in Lemma \ref{lem3.4}. Taking into account Lemma \ref{lem2.5}(3) and Lemma \ref{lem3.4}(1), we can conclude that $u_n$ is bounded in $\Upsilon_r$. Thus, up to a subsequence, there exists a $u_\mu \in \Upsilon_r$ such that
\begin{align*}
 &u_n \rightharpoonup u_\mu     \;\; \;\;\;\;\;\;\;\;\;\;\;\;\,  \mbox{in} \; \Upsilon_r , \\
&u_n \nabla u_n \rightharpoonup u \nabla u \;\;\;\;\,\mbox{in}   \;\; (L^2(\mathbb{R}^3))^3, \\
 &u_n \rightarrow  u_\mu        \;\;\;\;\;\;\;\;\;\;\;\;\;\;\;   \mbox{in} \;L^q(\mathbb{R}^3)\; \mbox{for}\; \mbox{all}\; q \in (2,6), \\
 &u_n \rightarrow  u_\mu\geq0   \;\;\;\;\;\;\;\;  \,             \mbox{a.e.} \; \mbox{in} \; \mathbb{R}^3.
\end{align*}
Moreover, by Lemma \ref{lem02-2} we obtain
\begin{equation}\label{eq3-7}
\int_{\mathbb{R}^3} h(u_n)u_n \rightarrow \int_{\mathbb{R}^3} h(u_\mu)u_\mu,
\end{equation}
\begin{equation}\label{eq3-8}
\int_{\mathbb{R}^3} H(u_n) \rightarrow \int_{\mathbb{R}^3} H(u_\mu) ,
\end{equation}
\begin{equation}\label{eq3-9}
\int_{\mathbb{R}^3} \tilde{H}(u_n) \rightarrow \int_{\mathbb{R}^3} \tilde{H}(u_\mu),
\end{equation}
 as $n \rightarrow \infty.$
Assume $u_\mu=0$, then
\begin{align*}
\mu(1+\gamma_\theta) \int_{\mathbb{R}^3} |\nabla u_n|^{\theta} +  \int_{\mathbb{R}^3} |\nabla u_n|^2 + 5\int_{\mathbb{R}^3} |u_n|^2 |\nabla u_n|^2
=P_\mu(u_n) +\frac{3}{2} \int_{\mathbb{R}^3} \tilde{H}(u_n) \rightarrow 0,
\end{align*}
which implies that $I_\mu(u_n) \rightarrow 0$, in contradiction with $m_\mu(a) >0$. Thus, $u_\mu \neq0.$
By \cite[Lemma 3]{BL1983} and Lemma \ref{lem3.4}(2), we conclude that there exists a sequence $\lambda_n \in \mathbb{R}$ such that
\begin{equation}\label{eq3.7}
I'_\mu(u_n) +\lambda_n u_n \rightarrow 0 \;\; \mbox{in} \; \Upsilon^*.
\end{equation}
Hence $\lambda_n=-\frac{1}{a^2}I'_\mu(u_n)[u_n]+o_n(1)$ is bounded in $\mathbb{R}$. We assume, up to a subsequence, $\lambda_n \rightarrow \lambda_\mu$. Since $u_n$ is bounded, we conclude
\begin{equation}\label{eq3.8}
I'_\mu(u_n) +\lambda_\mu u_n \rightarrow 0 \;\; \mbox{in} \; \Upsilon^*.
\end{equation}
As a result,
\begin{equation}\label{eq3.9}
I'_\mu(u_\mu) +\lambda_\mu u_\mu = 0 \;\; \mbox{in} \; \Upsilon^*.
\end{equation}

(2) By \eqref{eq3.9}, it is easy to see that $P_\mu(u_\mu)=0$.

(3) It follows from (2) that
$$ P_\mu(u_n) +\frac{3}{2} \int_{\mathbb{R}^3} \tilde{H}(u_n) \rightarrow P_\mu(u_\mu) +\frac{3}{2} \int_{\mathbb{R}^3} \tilde{H}(u_\mu) .$$
Taking into account the weak lower semicontinuous property \cite[Lemma 4.3]{CJS2010}, we have that
\begin{equation}\label{eq3.10}
\mu \int_{\mathbb{R}^3} |\nabla u_n|^{\theta} \rightarrow \mu \int_{\mathbb{R}^3} |\nabla u_\mu|^{\theta},
\end{equation}
\begin{equation}\label{eq3.11}
 \int_{\mathbb{R}^3} |\nabla u_n|^2 \rightarrow \int_{\mathbb{R}^3} |\nabla u_\mu|^2,
\end{equation}
\begin{equation}\label{eq3.12}
 \int_{\mathbb{R}^3} |u_n|^2 |\nabla u_n|^2 \rightarrow \int_{\mathbb{R}^3} |u_\mu|^2 |\nabla u_\mu|^2.
\end{equation}
Therefore,
$$ I_\mu(u_\mu)=\lim_{n\rightarrow \infty} I_\mu(u_n)=m_\mu(a).$$

(4) By \eqref{eq3-7} and \eqref{eq3.10}$-$\eqref{eq3.12}, we get
\begin{equation}\label{eq3.13}
I'_\mu(u_n)[u_n] \rightarrow I'_\mu(u_\mu)[u_\mu].
\end{equation}
Thus combining \eqref{eq3.8} with \eqref{eq3.9} and \eqref{eq3.13}, we conclude $\lambda_\mu \|u_n\|^2_2 \rightarrow \lambda_\mu \|u_\mu\|^2_2.$
Therefore, if $\lambda_\mu \neq0$, then
\begin{equation}\label{eq3.14}
u_n \rightarrow u_\mu,\ \|u_\mu\|_2=a.
\end{equation}

\end{proof}

\begin{lemma}\label{lem3.7}
Assume $h$ satisfies $(h_6)$, $\mu \in [0,1]$, $u\neq0$ and there exists a $\lambda \in \mathbb{R}$ such that $I'_\mu(u)+\lambda u=0$ in $\Upsilon^*.$ Then $\lambda >0$.
\end{lemma}

\begin{proof}
It is easy to see that
\begin{align*}
\int_{\mathbb{R}^3}h(u)u-\lambda \int_{\mathbb{R}^3} |u|^2 &=\mu \int_{\mathbb{R}^3} |\nabla u|^{\theta} + \int_{\mathbb{R}^3} |\nabla u|^2 + 4\int_{\mathbb{R}^3} |u|^2 |\nabla u|^2  \\
&< \mu(1+\gamma_\theta) \int_{\mathbb{R}^3} |\nabla u|^{\theta} + \int_{\mathbb{R}^3} |\nabla u|^2 + 5\int_{\mathbb{R}^3} |u|^2 |\nabla u|^2 \\
&= \frac{3}{2} \int_{\mathbb{R}^3} \left[h(u)u-2H(u)\right].
\end{align*}
By $(h_6)$, we obtain
$$ \lambda \int_{\mathbb{R}^3} |u|^2 > \frac{1}{2} \int_{\mathbb{R}^3} [6H(u)-h(u)u] >0.$$
As a result, $\lambda >0$.

\end{proof}

\section{Proof of Theorem \ref{Thm1.1}}\label{sec4}

In this section, we first establish a technical lemma, which means that the sequence of the critical points of $I_\mu|_{\mathcal{S}(a)}$ converges to the critical point of $I|_{\mathcal{S}(a)}$ as $\mu \rightarrow 0^+$.

\begin{lemma}\label{lem4.1}
Let $h$ satisfy $(h_1)-(h_4)$ and $(h_5')$, $\mu_n \rightarrow 0^+$, there exist $\lambda_{\mu_n} \geq 0$ and $u_{\mu_n} \in \mathcal{S}_r(a_n)$ with $0<a_n\leq a$ such that $I'_{\mu_n}(u_{\mu_n})+\lambda_{\mu_n} u_{\mu_n}=0$, $I_{\mu_n}(u_{\mu_n}) \rightarrow c>0$. Then  \\
(1) there exists a subsequence $u_{\mu_n}\rightharpoonup u$ in $H^1(\mathbb{R}^3)$ with $u\in H^1_r(\mathbb{R}^3) \cap L^{\infty}(\mathbb{R}^3)$;  \\
(2) there exists a $\lambda \in \mathbb{R}$ such that $\Theta(u,\lambda)=0$, $I(u)=c$ and $0<\|u\|_2\leq a$;  \\
(3) if $u_{\mu_n}\geq0$ for any $n \in \mathbb{N}^+$, then $u\geq0$;  \\
(4) if $\lambda\neq0$, then $\|u\|_2=\lim_{n\rightarrow \infty}a_n$.
\end{lemma}

\begin{proof}
(1) It follows from $I'_{\mu_n}(u_{\mu_n})+\lambda_{\mu_n} u_{\mu_n}=0$ that $P_{\mu_n}(u_{\mu_n})=0$ for any $n \in \mathbb{N}^+$. By Lemma \ref{lem2.5}(3), $I_{\mu_n}(u_{\mu_n}) \rightarrow c>0$ implies that $u_{\mu_n}$ is bounded in $H^1(\mathbb{R}^3)$ and $\int_{\mathbb{R}^3} |u_{\mu_n}|^2 |\nabla u_{\mu_n}|^2 \leq C$. Thus, up to a subsequence, there exists a $u \in H^1(\mathbb{R}^3)$ such that
\begin{align*}
 &u_{\mu_n} \rightharpoonup u     \;\; \;\;\;\;\;\;\;\;\;\;\;\;\;\;\;\;\;  \mbox{in} \; H^1_r(\mathbb{R}^3) , \\
 &u_{\mu_n} \nabla u_{\mu_n} \rightharpoonup u \nabla u \;\;\;\;\,\,\mbox{in}   \; (L^2(\mathbb{R}^3))^3, \\
 &u_{\mu_n} \rightarrow  u        \;\;\;\;\;\;\;\;\;\;\;\;\;\;\;\;\;\;\;   \mbox{in} \;L^q(\mathbb{R}^3)\;\, \mbox{for}\; \mbox{all}\; q \in (2,6), \\
 &u_{\mu_n} \rightarrow  u   \;\;\;\;\;\;\;\;\;\;\;\;\;\;\;\;\;\;\;          \mbox{a.e.} \; \mbox{in} \; \mathbb{R}^3.
\end{align*}
By Lemma \ref{lem02-2}, we obtain
\begin{equation}\label{eq4-1}
    \int_{\mathbb{R}^3} h(u_{\mu_n})u_{\mu_n} \rightarrow \int_{\mathbb{R}^3} h(u)u,
\end{equation}
\begin{equation}\label{eq4-2}
\int_{\mathbb{R}^3} H(u_{\mu_n}) \rightarrow \int_{\mathbb{R}^3} H(u) ,
\end{equation}
\begin{equation}\label{eq4-3}
\int_{\mathbb{R}^3} \tilde{H}(u_{\mu_n}) \rightarrow \int_{\mathbb{R}^3} \tilde{H}(u),
\end{equation}
 as $n \rightarrow \infty.$
Assume $a_n \rightarrow 0$, it follows from \eqref{GNI} that $u_{\mu_n} \rightarrow  0  \;\;   \mbox{in} \;L^{\frac{16}{3}}(\mathbb{R}^3)$, then $\int_{\mathbb{R}^3}  H(u_{\mu_n})\rightarrow 0$ and $\int_{\mathbb{R}^3}  \tilde{H}(u_{\mu_n})\rightarrow 0$. Taking into account $P_{\mu_n}(u_{\mu_n})=0$, we can conclude that $I_{\mu_n}(u_{\mu_n}) \rightarrow 0<c$. This is impossible. Hence, $\liminf_{n\rightarrow \infty}a_n>0$, and so $\lambda_{\mu_n}=-\frac{1}{a^2_n}I'_{\mu_n}(u_{\mu_n})[u_{\mu_n}]$ is bounded in $\mathbb{R}$. Thus, up to a subsequence, $\lambda_{\mu_n} \rightarrow \lambda$ in $ \mathbb{R}$. And it is easy to see that
\begin{equation}\label{eq4-4}
I'_{\mu_n}(u_{\mu_n}) +\lambda u_{\mu_n} \rightarrow 0 \;\; \mbox{in} \; \Upsilon^*.
\end{equation}

{\bf Claim:} There exists a constant $C>0$ such that $\|u_{\mu_n}\|_{\infty} \leq C$ and $\|u\|_{\infty} \leq C$.

{\bf Case 1:} If $\|u_{\mu_n}\|_\infty \leq 1$, then Claim is obvious.

{\bf Case 2:} If Case 1 is impossible, let
\begin{align*}
\begin{split}
w_n:=
\left \{
      \begin{array}{ll}
           u_{\mu_n},  &|u_{\mu_n}| > 1, \\
           0,           &|u_{\mu_n}(x)| \leq1. \\
      \end{array}
\right.
\end{split}
\end{align*}
Set $M>1, r>0$ and
\begin{align*}
v_n:=\max\{-M,\min\{w_n, M\}\}.
\end{align*}
Let $\varphi=w_n|v_n|^{2r}$, then $\varphi \in \Upsilon$. By $I'_{\mu_n}(u_{\mu_n})+\lambda_{\mu_n} u_{\mu_n}=0$ and $\lambda_{\mu_n} \geq 0$, we can infer that
\begin{align*}
&\ \int_{\mathbb{R}^3}h(w_n)\varphi   \\
=&\ \mu_n \int_{\mathbb{R}^3} |\nabla w_n|^{\theta-2}\nabla w_n \cdot \nabla \varphi
+\int_{\mathbb{R}^3} \nabla w_n \cdot \nabla \varphi
+2\int_{\mathbb{R}^3}\left(|w_n|^2 \nabla w_n \cdot \nabla \varphi +w_n\varphi|\nabla w_n|^2\right)  \\
&\,+\lambda_{\mu_n} \int_{\mathbb{R}^3}w_n\varphi  \\
\geq &\ 2\int_{\mathbb{R}^3} |w_n|^2\nabla w_n \cdot \nabla \varphi  \\
=&\ 2\int_{\mathbb{R}^3}\left( |w_n|^2|\nabla w_n|^2|v_n|^{2r} +2r|w_n|^2|v_n|^{2r-2}w_nv_n\nabla w_n\cdot \nabla v_n \right)  \\
=&\ \frac{1}{2} \int_{\mathbb{R}^3} \left||v_n|^r \nabla (|w_n|^2)\right|^2 +\frac{4}{r}\int_{\mathbb{R}^3}\left| |w_n|^2\nabla (|v_n|^r)\right|^2   \\
\geq&\ \frac{1}{r+4} \int_{\mathbb{R}^3} \left| \nabla (|w_n|^2|v_n|^r)\right|^2   \\
\geq&\ \frac{C}{(r+2)^2} \left[\int_{\mathbb{R}^3} \left(|w_n|^2|v_n|^r\right)^{6}\right]^{\frac{1}{3}}.
\end{align*}
It is easy to see that $|h(w_n)| \leq C\left(|w_n|^{\frac{13}{3}} + |w_n|^{5}\right) \leq C|w_n|^{5}$. Taking into account the interpolation inequality, we get
\begin{align*}
\int_{\mathbb{R}^3}h(w_n)\varphi   &\leq C\int_{\mathbb{R}^3} |w_n|^{5}\varphi   =C\int_{\mathbb{R}^3} |w_n|^{6}|v_n|^{2r}  \\
&\leq C\left(\int_{\mathbb{R}^3} |w_n|^{12}\right)^{\frac{1}{6}} \left[\int_{\mathbb{R}^3} \left(|v_n|^r|w_n|^2\right)^{\frac{12}{5}}\right]^{\frac{5}{6}}  \\
&\leq C \left[\int_{\mathbb{R}^3} \left(|v_n|^r|w_n|^2\right)^{\frac{12}{5}}\right]^{\frac{5}{6}}
\end{align*}
As a result,
\begin{equation}\label{eq4.1}
\left[\int_{\mathbb{R}^3} \left(|v_n|^r|w_n|^2\right)^{6}\right]^{\frac{1}{6}} \leq C(r+2) \left[\int_{\mathbb{R}^3} \left(|v_n|^r|w_n|^2\right)^{\frac{12}{5}}\right]^{\frac{5}{12}}.
\end{equation}
Let $r_0=3$, $d=\frac{5}{2}$. Taking $r=r_0$ in \eqref{eq4.1}, and letting $M \rightarrow +\infty$, we can derive
\begin{equation}\label{eq4.2}
\|w_n\|_{12d} \leq \ [C(r_0+2)]^\frac{1}{r_0+2}\|w_n\|_{12}.
\end{equation}
Set $r_{i+1}+2=(r_i+2)d$ for $i\in\mathbb{N}^+$. Then, by induction, we infer that
\begin{equation}\label{eq4.3}
\|w_n\|_{12d^{i+1}} \leq \ \prod^i_{k=0} [C(r_k+2)]^\frac{1}{r_k+2}\|w_n\|_{12}   \leq C\|w_n\|_{12}.
\end{equation}
Taking $i \rightarrow +\infty$ in \eqref{eq4.3}, we obtain
$$\|w_n\|_{\infty} \leq C.$$
Therefore,
\begin{equation}\label{eq4.4}
\|u_{\mu_n}\|_{\infty} \leq C,\ \|u\|_{\infty} \leq C.
\end{equation}

(2) Let $\varphi=\psi e^{-u_{\mu_n}}$ with $0\leq \psi \in \mathcal{C}^{\infty}_0(\mathbb{R}^3)$, we can derive
\begin{align*}
0=&\ (I'_{\mu_n}(u_{\mu_n})+\lambda_{\mu_n} u_{\mu_n})[\varphi] \\
 =&\ \mu_n \int_{\mathbb{R}^3} |\nabla u_{\mu_n}|^{\theta-2}\nabla u_{\mu_n} \cdot (\nabla \psi e^{-u_{\mu_n}}-\psi e^{-u_{\mu_n}}\nabla u_{\mu_n})       +\int_{\mathbb{R}^3} \nabla u_{\mu_n} \cdot (\nabla \psi e^{-u_{\mu_n}}-\psi e^{-u_{\mu_n}}\nabla u_{\mu_n})  \\
 &+2\int_{\mathbb{R}^3} |u_{\mu_n}|^2 \nabla u_{\mu_n} \cdot (\nabla \psi e^{-u_{\mu_n}}-\psi e^{-u_{\mu_n}}\nabla u_{\mu_n})
 +2\int_{\mathbb{R}^3} u_{\mu_n}\psi e^{-u_{\mu_n}} |\nabla u_{\mu_n}|^2
  \\
  &\,+\lambda_{\mu_n} \int_{\mathbb{R}^3} u_{\mu_n}\psi e^{-u_{\mu_n}} -\int_{\mathbb{R}^3} h(u_{\mu_n})\psi e^{-u_{\mu_n}}  \\
 \leq&\ \mu_n \int_{\mathbb{R}^3} |\nabla u_{\mu_n}|^{\theta-2}\nabla u_{\mu_n} \cdot \nabla \psi e^{-u_{\mu_n}}
 + \int_{\mathbb{R}^3}\left(1+2u_{\mu_n}^2\right)\nabla u_{\mu_n} \cdot \nabla \psi e^{-u_{\mu_n}}   \\
 &-\int_{\mathbb{R}^3}\left(1+2u_{\mu_n}^2-2u_{\mu_n}\right) \psi e^{-u_{\mu_n}}|\nabla u_{\mu_n}|^2   +\lambda_{\mu_n} \int_{\mathbb{R}^3} u_{\mu_n}\psi e^{-u_{\mu_n}}  \\
 &-\int_{\mathbb{R}^3} h(u_{\mu_n})\psi e^{-u_{\mu_n}}.
\end{align*}
By Lemma \ref{lem2.5}(3) and $\mu_n \rightarrow 0^+$, we infer that
\begin{align*}
     \left|\mu_n \int_{\mathbb{R}^3} |\nabla u_{\mu_n}|^{\theta-2}\nabla u_{\mu_n} \cdot \nabla \psi e^{-u_{\mu_n}} \right|
\leq C|\mu_n \|\nabla u_{\mu_n}\|^{\theta-1}_\theta|
=    C|\mu_n \|\nabla u_{\mu_n}\|^{\theta}_\theta|^{\frac{\theta-1}{\theta}} |\mu_n|^{\frac{1}{\theta}}  \rightarrow 0^+.
\end{align*}
Taking into account (1), the H\"older inequality, the weak convergence, and the Lebesgue's dominated convergence theorem, we can deduce that
$$\int_{\mathbb{R}^3}\left(1+2u_{\mu_n}^2\right)\nabla u_{\mu_n} \cdot \nabla \psi e^{-u_{\mu_n}} \rightarrow  \int_{\mathbb{R}^3}(1+2u^2)\nabla u \cdot \nabla \psi e^{-u} ,$$
$$\lambda_{\mu_n} \int_{\mathbb{R}^3} u_{\mu_n}\psi e^{-u_{\mu_n}}  \rightarrow \lambda \int_{\mathbb{R}^3} u\psi e^{-u} ,$$
$$\int_{\mathbb{R}^3} h(u_{\mu_n})\psi e^{-u_{\mu_n}} \rightarrow \int_{\mathbb{R}^3} h(u)\psi e^{-u}.$$
By the Fatou's lemma, we conclude that
$$\liminf_{n\rightarrow \infty} \int_{\mathbb{R}^3}(1+2u_{\mu_n}^2-2u_{\mu_n}) \psi e^{-u_{\mu_n}}|\nabla u_{\mu_n}|^2 \geq \int_{\mathbb{R}^3}(1+2u^2-2u) \psi e^{-u}|\nabla u|^2. $$
As a result,
\begin{align*}
0\leq & \int_{\mathbb{R}^3} \nabla u \cdot (\nabla \psi e^{-u}-\psi e^{-u}\nabla u)  +2\int_{\mathbb{R}^3} |u|^2 \nabla u \cdot (\nabla \psi e^{-u}-\psi e^{-u}\nabla u)  \\
 & +2\int_{\mathbb{R}^3} u\psi e^{-u} |\nabla u|^2 +\lambda \int_{\mathbb{R}^3} u\psi e^{-u}
 -\int_{\mathbb{R}^3} h(u)\psi e^{-u} .
\end{align*}
For any fixed $\phi \in \mathcal{C}^{\infty}_0(\mathbb{R}^3)$ with $\phi \geq0$, we take $\psi_n \in \mathcal{C}^{\infty}_0(\mathbb{R}^3)$ and $\psi_n \geq0$ such that
\begin{align*}
\psi_n \rightarrow \phi e^u \;\mbox{in} \; H^1(\mathbb{R}^3),\;\;  \psi_n \rightarrow \phi e^u \; \mbox{a.e.} \;\mbox{in} \; \mathbb{R}^3.
\end{align*}
Therefore,
\begin{align*}
0\leq & \int_{\mathbb{R}^3} \nabla u \cdot \nabla \phi   +2\int_{\mathbb{R}^3} \left(|u|^2 \nabla u \cdot \nabla \phi
  +u\phi |\nabla u|^2\right) +\lambda \int_{\mathbb{R}^3} u\phi
 -\int_{\mathbb{R}^3} h(u)\phi .
\end{align*}
Similarly by taking $\varphi=\psi e^{u_{\mu_n}}$, we get
\begin{align*}
0\geq & \int_{\mathbb{R}^3} \nabla u \cdot \nabla \phi   +2\int_{\mathbb{R}^3} \left(|u|^2 \nabla u \cdot \nabla \phi
  +u\phi |\nabla u|^2\right) +\lambda \int_{\mathbb{R}^3} u\phi
 -\int_{\mathbb{R}^3} h(u)\phi .
\end{align*}
It is easy to get
\begin{equation}\label{eq4-9}
    \Theta(u,\lambda)=0,
\end{equation}
which implies that $P(u)=0$.
Obviously,
$$P_{\mu_n}(u_{\mu_n}) +\frac{3}{2}\int_{\mathbb{R}^3}\tilde{H}(u_{\mu_n}) \rightarrow P(u) +\frac{3}{2}\int_{\mathbb{R}^3}\tilde{H}(u).$$
Similar to \eqref{eq3.10}$-$\eqref{eq3.12}, we have
\begin{equation}\label{eq4.5}
\mu_n \int_{\mathbb{R}^3} |\nabla u_{\mu_n}|^{\theta} \rightarrow 0,
\end{equation}
\begin{equation}\label{eq4.6}
 \int_{\mathbb{R}^3} |\nabla u_{\mu_n}|^2 \rightarrow \int_{\mathbb{R}^3} |\nabla u|^2,
\end{equation}
\begin{equation}\label{eq4.7}
 \int_{\mathbb{R}^3} |u_{\mu_n}|^2 |\nabla u_{\mu_n}|^2 \rightarrow \int_{\mathbb{R}^3} |u|^2 |\nabla u|^2.
\end{equation}
Then
$$I(u)=\lim_{n\rightarrow \infty} I_{\mu_n}(u_{\mu_n})=c.$$
Consequently,
$$0<\|u\|_2 \leq \liminf_{n\rightarrow \infty} \|u_{\mu_n}\|_2 \leq a.$$

(3) If $u_{\mu_n} \geq0$ for any $n \in \mathbb{N}^+$, then for any $x \in \mathbb{R}^3$, we have
$$u(x)=\lim_{n\rightarrow \infty} u_{\mu_n}(x) \geq0.$$

(4) It follows from \eqref{eq4-1} and \eqref{eq4.5}$-$\eqref{eq4.7} that
\begin{equation}\label{eq4.13}
    I'_{\mu_n}(u_{\mu_n})[u_{\mu_n}] \rightarrow \int_{\mathbb{R}^3} |\nabla u|^2  +4\int_{\mathbb{R}^3} |u|^2 |\nabla u|^2
 -\int_{\mathbb{R}^3} h(u)u.
\end{equation}
Thus combining \eqref{eq4-4} with \eqref{eq4-9} and \eqref{eq4.13}, we obtain
$\lambda \|u_{\mu_n}\|^2_2 \rightarrow \lambda \|u\|^2_2$.
If $\lambda \neq0$, then
$$\|u\|_2=\lim_{n\rightarrow \infty}a_n.$$
\end{proof}

\begin{remark}
We work in a radial space which yields the compact embedding. As for non-radial space, the authors \cite{JL2020} provided a method to handle the compactness by analyzing the nonincreasing property of $m \mapsto E_m$.
And we can establish a similar property for $a \mapsto m_\mu(a)$ to conclude Lemma \ref{lem3.5}. However, it is impossible to derive Lemma \ref{lem4.1} from this property.
\end{remark}

Based on the above preliminary works, we are able to finish the proof of Theorem \ref{Thm1.1}.

\begin{proof} [Proof of Theorem \ref{Thm1.1}]
For any $0<\mu_1<\mu_2 \leq1$, since $I_{\mu_1}(u) \leq I_{\mu_2}(u)$ and $\Gamma_{\mu_1} \supset \Gamma_{\mu_2}$, there holds
\begin{align*}
m_{\mu_1}(a) &=     \inf_{\gamma \in \Gamma_{\mu_1}} \sup_{t \in [0,1]} \Phi_{\mu_1}(\gamma(t)) \leq  \inf_{\gamma \in \Gamma_{\mu_2}} \sup_{t \in [0,1]} \Phi_{\mu_1}(\gamma(t))  \\
             &\leq  \inf_{\gamma \in \Gamma_{\mu_2}} \sup_{t \in [0,1]} \Phi_{\mu_2}(\gamma(t))  =     m_{\mu_2}(a) ,
\end{align*}
i.e., $m_\mu(a)$ is nondecreasing with respect to $\mu \in (0,1]$. From Remark \ref{rem3.5}, we infer that
$$b(a):=\lim_{\mu \rightarrow 0^+} m_\mu(a) \geq \frac{\zeta}{8}>0.$$
It follows from Lemma \ref{lem3.5} that there exists a sequence $\{\mu_n\}$ such that
$$ I'_{\mu_n}(u_{\mu_n})+\lambda_{\mu_n}u_{\mu_n}=0, \;\; I_{\mu_n}(u_{\mu_n})=m_{\mu_n}(a) \rightarrow b(a), \;\mbox{as} \; \mu_n \rightarrow 0^+,$$
where $u_{\mu_n} \in \mathcal{S}_r(a_n)$ with $0<a_n\leq a$ and $u_{\mu_n} \geq0$. Taking into account Lemma \ref{lem3.7}, we have $\lambda_{\mu_n}>0$. By Lemma \ref{lem3.5}, we conclude that $a_n=\|u_{\mu_n}\|_2=a$.    At the same time, Lemma \ref{lem4.1} implies that there exists a $v\geq0,\; v \in H^1_r(\mathbb{R}^3)\cap L^{\infty}(\mathbb{R}^3)$ and a $\lambda_0 \in \mathbb{R}$ such that
$$\Theta(v,\lambda_0)=0, \;\; I(v)=b(a),\;\; 0<\|v\|_2\leq a.$$
So, by Lemma \ref{lem3.7}, we get $\lambda_0>0$, which means $\|v\|_2=\lim_{n\rightarrow \infty}a_n =a$. Obviously, $v$ is a nontrivial nonnegative solution of \eqref{eq01} and \eqref{eq02}.

Recall
$$\sigma(a)=\inf_{u \in \mathcal{S}'(a), \Theta(u,\lambda)=0} I(u).$$
Obviously, $\sigma(a) \leq I(v)=b(a)$. Further, using the similar approach to Lemma \ref{lem2.5}(1), we conclude that $\sigma(a)>0$.

We take a sequence $v_n \neq0$ and $v_n \geq0$ such that
$$ v_n \in \mathcal{S}'(a),\;  \Theta(v_n,\lambda_n)=0,\;  I(v_n) \rightarrow \sigma(a).$$
By the similar approach to Lemma \ref{lem4.1}, we conclude that
there exists a $u \neq0,\; u\geq0,\; u \in H^1_r(\mathbb{R}^3)\cap L^{\infty}(\mathbb{R}^3)$ and a $\lambda \in \mathbb{R}$ such that
$$\Theta(u,\lambda)=0, \;\; I(u)=\sigma(a).$$
It follows from Lemma \ref{lem3.7} that $\lambda >0$. As a result, $\|u\|_2=a$. By \cite[Lemma 2.6]{LLW2013-2}  and $u \in L^{\infty}(\mathbb{R}^3)$, we get $u>0$.
\end{proof}

\section{Proof of Theorem \ref{Thm1.2}}\label{sec5}

Let $Y \subset \Upsilon$. A set $A \subset Y$ is called $\tau$-invariant if $\tau(A)=A$. A homotopy $\eta:[0,1]\times Y \mapsto Y$ is $\tau$-equivariant if $\eta(t,\tau(u))=\tau(\eta(t,u))$ for any $(t,u) \in [0,1]\times Y$.

\begin{definition}\label{def5.1}
\cite[Definition 7.1]{G1993} Let B be a closed $\tau$-invariant subset of Y. We say that a class $\mathcal{F}$ of compact subsets of Y is a $\tau$-homotopy stable family with boundary B provided  \\
(a) every set in $\mathcal{F}$ is $\tau$-invariant;  \\
(b) every set in $\mathcal{F}$ contains B;  \\
(c) for any set A in $\mathcal{F}$ and any $\tau$-equivariant homoyopy $\eta \in \mathcal{C}([0,1]\times Y,Y)$ satisfying $\eta(t,x)=x$ for all $(t,x)$ in $(\{0\}\times Y)\cup ([0,1]\times B) $ we have that $\eta(1,A) \in \mathcal{F}.$
\end{definition}

We consider an auxiliary functional as the one in \cite{JL2020}:
\begin{align*}
&\ \Psi_\mu(u):=I_\mu(s(u)*u)   \\
=&\ \frac{\mu}{\theta} e^{\theta(1+\gamma_\theta)s(u)}  \int_{\mathbb{R}^3} |\nabla u|^{\theta} + \frac{1}{2}e^{2s(u)} \int_{\mathbb{R}^3} |\nabla u|^2 + e^{5s(u)}\int_{\mathbb{R}^3} |u|^2 |\nabla u|^2 -e^{-3s(u)}\int_{\mathbb{R}^3}  H(e^{\frac{3}{2}s(u)}u) ,
\end{align*}
where $s(u)$ is given by Lemma \ref{lem2.4}. Taking $\tau(u)=-u$, then $\Psi_\mu(u)$ is $\tau$-invariant.

\begin{lemma}\label{lem5.2}
Let $h$ satisfy $(h_1)-(h_4)$ and $(h_5')$. Then for $\mu \in (0,1]$, $u \in \Upsilon \backslash \{0\}$ and $\varphi \in \Upsilon$, we have
\begin{align*}
\Psi'_\mu(u)[\varphi]=&\ \mu e^{\theta(1+\gamma_\theta)s(u)} \int_{\mathbb{R}^3} |\nabla u|^{\theta-2}\nabla u \cdot \nabla \varphi       +e^{2s(u)}\int_{\mathbb{R}^3} \nabla u \cdot \nabla \varphi  \\
 &+2e^{5s(u)}\int_{\mathbb{R}^3} |u|^2 \nabla u \cdot \nabla\varphi + u\varphi |\nabla u|^2 -e^{-\frac{3}{2}s(u)}\int_{\mathbb{R}^3} h\left(e^{\frac{3}{2}s(u)}u\right)\varphi  \\
 =&\ I'_\mu(s(u)*u)[s(u)*\varphi].
\end{align*}
\end{lemma}

\begin{proof}
Let $u \in \Upsilon \backslash \{0\}$, $\varphi \in \Upsilon$, $u_t=u+t\varphi$ and $s_t=s(u_t)$ with $|t|$ small enough. It follows from the mean value theorem that
\begin{align*}
\Psi_\mu(u_t)-\Psi_\mu(u)
=&\ I_\mu(s_t*u_t)-I_\mu(s_0*u) \leq I_\mu(s_t*u_t)-I_\mu(s_t*u)  \\
  =&\ \mu e^{\theta(1+\gamma_\theta)s_t} \int_{\mathbb{R}^3} \left(|\nabla u_{\xi_t}|^{\theta-2}\nabla u_{\xi_t} \cdot \nabla \varphi\right)  t +e^{2s_t}\int_{\mathbb{R}^3} \left(\nabla u_{\xi_t} \cdot \nabla \varphi\right) t  \\
&+2e^{5s_t}\int_{\mathbb{R}^3} \left(|u_{\xi_t}|^2 \nabla u_{\xi_t} \cdot \nabla \varphi
 + u_{\xi_t}\varphi |\nabla u_{\xi_t}|^2\right)t -e^{-\frac{3}{2}s_t}\int_{\mathbb{R}^3} \left[h\left(e^{\frac{3}{2}s_t}u_{\xi_t}\right)\varphi\right] t,
\end{align*}
where $|\xi_t| \in (0,|t|).$
On the other hand,
\begin{align*}
 \Psi_\mu(u_t)-\Psi_\mu(u) =&\ I_\mu(s_t*u_t)-I_\mu(s_0*u) \geq I_\mu(s_0*u_t)-I_\mu(s_0*u)  \\
  =&\ \mu e^{\theta(1+\gamma_\theta)s_0} \int_{\mathbb{R}^3} \left(|\nabla u_{\kappa_t}|^{\theta-2}\nabla u_{\kappa_t} \cdot \nabla \varphi\right) t +e^{2s_0}\int_{\mathbb{R}^3} \left(\nabla u_{\kappa_t} \cdot \nabla\varphi\right) t  \\
&+2e^{5s_0}\int_{\mathbb{R}^3} \left(|u_{\kappa_t}|^2 \nabla u_{\kappa_t} \cdot \nabla \varphi
 + u_{\kappa_t}\varphi |\nabla u_{\kappa_t}|^2\right)t -e^{-\frac{3}{2}s_0}\int_{\mathbb{R}^3} \left[h\left(e^{\frac{3}{2}s_0}u_{\kappa_t}\right)\varphi \right]t,
\end{align*}
where $|\kappa_t| \in (0,|t|).$

Since $\lim_{t \rightarrow 0} s_t = s_0$, we obtain that
\begin{align*}
\Psi'_\mu(u)[\varphi]=&\ \lim_{t \rightarrow 0} \frac{\Psi_\mu(u_t)-\Psi_\mu(u)}{t}  \\
=&\ \mu e^{\theta(1+\gamma_\theta)s(u)} \int_{\mathbb{R}^3} |\nabla u|^{\theta-2}\nabla u \cdot \nabla \varphi       +e^{2s(u)}\int_{\mathbb{R}^3} \nabla u \cdot \nabla \varphi  \\
 &+2e^{5s(u)}\int_{\mathbb{R}^3} |u|^2 \nabla u \cdot \nabla\varphi + u\varphi |\nabla u|^2 -e^{-\frac{3}{2}s(u)}\int_{\mathbb{R}^3} h\left(e^{\frac{3}{2}s(u)}u\right)\varphi .
\end{align*}
By changing variables in the integrals, we can get that
$$\Psi'_\mu(u)[\varphi]=I'_\mu(s(u)*u)[s(u)*\varphi].$$
\end{proof}

It is easy to see that the functional $\Psi_\mu$ is of class $\mathcal{C}^1$. Similar to \cite[Lemma 3.12]{LZ2023} and by Lemma \ref{lem5.2}, we have the next lemma.

\begin{lemma}\label{lem5.3}
Assume $h$ satisfies $(h_1)-(h_5)$. Let $\mathcal{F}$ is a $\tau$-homotopy stable family of compact subsets of $Y=\mathcal{S}_r(a)$ with boundary $B = \emptyset$, and set
$$d:=\inf_{A \in \mathcal{F}} \max_{u \in A} \Psi_\mu(u).$$
If $d>0$, then there exists a sequence $u_n \in S_r(a)$ such that
$$I_\mu(u_n) \rightarrow d,\; I_\mu|'_{\mathcal{S}(a)}(u_n) \rightarrow 0,\; P_\mu(u_n)=0.$$
\end{lemma}

We recall the definition of the genus of $\tau$-invariant sets due to Krasnoselskii and refer the readers to \cite{R1986}.
\begin{definition}
For any nonempty closed $\tau$-invariant set $A \subset \Upsilon$, the genus of A is defined by
\begin{align*}
    {\rm gen}(A):=\min\{k \in \mathbb{N}^+\;| \;\exists\, \phi:A \mapsto \mathbb{R}^k \backslash \{0\}, \phi \;is \; odd \; and \; continuous\}.
\end{align*}
\end{definition}
We take
\begin{align*}
  \begin{split}
     \mbox{gen}(A)=\left \{
        \begin{array}{ll}
            +\infty,    &\mbox{finite} \; k \; \mbox{does} \; \mbox{not}\; \mbox{exist}, \\
            0,          &A =\emptyset.
        \end{array}
     \right.
  \end{split}
\end{align*}
Let $\mathcal{A}(a)$ be the family of compact $\tau$-invariant subsets of $\mathcal{S}_r(a)$. For any $j \in \mathbb{N}^+$, we define
$$\mathcal{A}_j(a):=\{A \in \mathcal{A}(a)\;|\; \mbox{gen}(A) \geq j\}$$
and
$$c^j_\mu(a):=\inf_{A \in \mathcal{A}_j(a)} \max_{u \in A} \Psi_\mu(u).$$

By a similar approach to  \cite[Lemma 3.13]{LZ2023}, we obtain the next lemma.
\begin{lemma}\label{lem5.5}
Let $h$ satisfy $(h_1)-(h_5)$. Then we have   \\
(1) $\mathcal{A}_j(a) \neq \emptyset$ for any $j\in \mathbb{N}^+$, and $\mathcal{A}_j(a)$ is a $\tau$-homotopy stable family of compact subsets of $\mathcal{S}_r(a)$ with boundary $B = \emptyset$;   \\
(2) $c^{j+1}_\mu(a) \geq c^j_\mu(a) \geq \frac{\zeta}{8} >0$ for any $\mu \in (0,1]$ and $j\in \mathbb{N}^+$;  \\
(3) $c^j_\mu(a)$ is nondecreasing with respect to $\mu \in (0,1]$ for any $j\in \mathbb{N}^+$;   \\
(4) $\tilde{b}^j(a):= \inf_{0<\mu \leq 1} c^j_\mu(a) \rightarrow +\infty$ as $j \rightarrow +\infty$.
\end{lemma}

For any fixed $\mu \in (0,1]$ and any $j\in \mathbb{N}^+$, taking into account Lemma \ref{lem5.3} and Lemma \ref{lem5.5}, we may get a sequence $u_n \in \mathcal{S}_r(a)$ such that
$$I_\mu(u_n) \rightarrow c^j_\mu(a),\; I_\mu|'_{\mathcal{S}(a)}(u_n) \rightarrow 0,\;  P_\mu(u_n)=0,$$
as $n \rightarrow +\infty$.

Therefore, similar to Lemma \ref{lem3.5}, we obtain the next lemma.
\begin{lemma}\label{lem5.6}
If $h$ satisfies $(h_1)-(h_5)$, then for any fixed $\mu \in (0,1]$ and any $j\in \mathbb{N}^+$, there exist a  $u^j_\mu \in \Upsilon_r \backslash \{0\}$ and a $\lambda^j_\mu \in \mathbb{R}$ such that   \\
(1) $I'_\mu(u^j_\mu) +\lambda^j_\mu u^j_\mu=0;$  \\
(2) $P_\mu(u^j_\mu)=0;$  \\
(3) $I_\mu(u^j_\mu)=c^j_\mu(a);$  \\
(4) $0<\|u^j_\mu\|_2 \leq a;$ furthermore, if $\lambda^j_\mu \neq 0$, then $\|u^j_\mu\|_2 = a.$
\end{lemma}

Now, we can complete the proof of Theorem \ref{Thm1.2}.

\begin{proof}[Proof of Theorem \ref{Thm1.2}]
It follows from Lemma \ref{lem5.5} that
$$\tilde{b}^j(a)= \lim_{\mu \rightarrow 0^+} c^j_\mu(a)\geq \frac{\zeta}{8}>0,\; \tilde{b}^j(a) \rightarrow +\infty.$$
By Lemma \ref{lem5.6}, for any $j\in \mathbb{N}^+$, there exists a sequence $\{\mu^j_n\}$ such that
$$\mu^j_n \rightarrow 0^+,\; I'_{\mu^j_n}\left(u^j_{\mu^j_n}\right) +\lambda^j_{\mu^j_n} u^j_{\mu^j_n}=0,\; I_{\mu^j_n}\left(u^j_{\mu^j_n}\right)=c^j_{\mu^j_n}(a) \rightarrow \tilde{b}^j(a),$$
as $n \rightarrow +\infty$, where $u^j_{\mu^j_n} \in \mathcal{S}_r(a^j_n)$ with $0<a^j_n \leq a$. From Lemma \ref{lem3.7}, we have $\lambda^j_{\mu^j_n}>0$, which means $a^j_n=\left\|u^j_{\mu^j_n}\right\|_2=a$.
As a result, Lemma \ref{lem4.1} implies that there exist a $u^j \in H^1_r(\mathbb{R}^3) \cap L^\infty(\mathbb{R}^3)$ and a $\lambda^j \in \mathbb{R}$ such that
$$\Theta(u^j,\lambda^j)=0,\; I(u^j)=\tilde{b}^j(a),\; 0<\|u^j\|_2\leq a.$$
It follows from Lemma \ref{lem3.7} that $\lambda^j>0$. Therefore, $\|u^j\|_2=\lim_{n\rightarrow \infty}a^j_n =a $. Moreover, $I(u^j)=\tilde{b}^j(a) \rightarrow +\infty$.
\end{proof}

\section{Proof of Theorem \ref{Thm1.5}}\label{sec6}

We split the proof of Theorem \ref{Thm1.5} into several Lemmas.

\begin{lemma}\label{lem6.3}
    Let $(h_1)-(h_4)$ and $(h_5')$ hold, then the function $a \mapsto m(a)$ is lower semicontinuous with respect to any $a>0$.
\end{lemma}

\begin{proof}
It suffices to prove that for any positive sequence $\{a_n\}$ with $a_n \rightarrow a$ as $n \rightarrow \infty$, we have $m(a) \leq \liminf_{n \rightarrow \infty}m(a_n)$. We divide the proof into two steps.

\textbf{Step 1:}
We prove that $\limsup_{n\rightarrow \infty}m(a_n) \leq C$.

Fix $u \in \mathcal{S}'(a) \cap L^{\infty}(\mathbb{R}^3)$, set $u_n:=\frac{a_n}{a}u  \in \mathcal{S}'(a_n)$. It is clear that $\int_{\mathbb{R}^3} |u_n|^2 |\nabla u_n|^2 \rightarrow \int_{\mathbb{R}^3} |u|^2 |\nabla u|^2$ and $u_n \rightarrow u$ in $H^1(\mathbb{R}^3)$. By Lemma \ref{lem2.4}, there exists a unique $s(u_n) \in \mathbb{R}$ such that $s(u_n)*u_n \in \mathcal{P}'(a_n)$ and $\lim_{n \rightarrow \infty}s(u_n) \rightarrow s(u).$
And thus
 $$s(u_n)*u_n \rightarrow s(u)*u\;\,\mbox{in} \; H^1(\mathbb{R}^3),$$
 $$\int_{\mathbb{R}^3} |s(u_n)*u_n|^2 |\nabla (s(u_n)*u_n)|^2 \rightarrow \int_{\mathbb{R}^3} |s(u)*u|^2 |\nabla (s(u)*u)|^2.$$
As a consequence,
$$\limsup_{n\rightarrow \infty}m(a_n) \leq
\limsup_{n\rightarrow \infty}I(s(u_n)*u_n) =
I(s(u)*u).$$

\textbf{Step 2:}
We show that $m(a)  \leq \liminf_{n\rightarrow \infty}m(a_n)$.

In fact, for each $n \in \mathbb{N}^+$, from the definition of the infimum, there exists a sequence $v_n \in \mathcal{P}'(a_n)$ such that
$$I(v_n) \leq m(a_n)+\frac{1}{n}.$$
Let $t_n:=\left(\frac{a}{a_n}\right)^\frac{2}{3} \rightarrow 1$ as $n \rightarrow \infty$ and $\tilde{v}_n(x):=v_n(\frac{x}{t_n}) \in \mathcal{S}'(a)$. It is easy to see that
\begin{align*}
m(a) \leq&\ I(s(\tilde{v}_n)*\tilde{v}_n)  \\
        \leq&\ I(s(\tilde{v}_n)*\tilde{v}_n)-I(s(\tilde{v}_n)*v_n)+I(s(\tilde{v}_n)*v_n)   \\
        \leq&\ |I(s(\tilde{v}_n)*\tilde{v}_n)-I(s(\tilde{v}_n)*v_n)|+I(v_n)  \\
        \leq&\ |I(s(\tilde{v}_n)*\tilde{v}_n)-I(s(\tilde{v}_n)*v_n)| +m(a_n)+\frac{1}{n}.
\end{align*}
Clearly,
\begin{align*}
    &\ |I(s(\tilde{v}_n)*\tilde{v}_n)-I(s(\tilde{v}_n)*v_n)|    \\
    =&\ \bigg|\frac{1}{2}(t_n-1)\int_{\mathbb{R}^3} |\nabla (s(\tilde{v}_n)*v_n)|^2 + (t_n-1)\int_{\mathbb{R}^3} |s(\tilde{v}_n)*v_n|^2 |\nabla (s(\tilde{v}_n)*v_n)|^2   \\
 &\,-(t_n^3-1)\int_{\mathbb{R}^3}H(s(\tilde{v}_n)*v_n)\bigg|  \\
    \leq&\ \frac{1}{2}|t_n-1|\int_{\mathbb{R}^3} |\nabla (s(\tilde{v}_n)*v_n)|^2 + |t_n-1|\int_{\mathbb{R}^3} |s(\tilde{v}_n)*v_n|^2 |\nabla (s(\tilde{v}_n)*v_n)|^2   \\
 &\,+|t_n^3-1|\int_{\mathbb{R}^3}H(s(\tilde{v}_n)*v_n).
\end{align*}
By Lemma \ref{lem2.5}(3) and Step 1, we can conclude that $\int_{\mathbb{R}^3} |v_n|^2 |\nabla v_n|^2 \leq C$ and $\{v_n\}$ is bounded in $H^1(\mathbb{R}^3)$, which combining with $t_n \rightarrow 1$ as $n \rightarrow \infty$ we see that $\int_{\mathbb{R}^3} |\tilde{v}_n|^2 |\nabla \tilde{v}_n|^2 \leq C$ and $\{\tilde{v}_n\}$ is bounded in $H^1(\mathbb{R}^3)$.

\textbf{Claim 1:}
There exists a sequence  $\{\tilde{y}_n\} \subset \mathbb{R}^3$ and $\tilde{v} \in H^1(\mathbb{R}^3)$ such that up to a subsequence $\tilde{v}_n(\cdot +\tilde{y}_n) \rightarrow \tilde{v}\neq 0$ a.e. in $\mathbb{R}^3$.

Set
$$\tilde{\rho} :=\limsup_{n \rightarrow \infty} \sup_{\tilde{y} \in \mathbb{R}^3} \int_{B_1(\tilde{y})} |\tilde{v}_n|^2.$$
 If $\tilde{\rho}=0$, then $\tilde{v}_n \rightarrow 0$ in $L^{\frac{16}{3}}(\mathbb{R}^3)$ by \cite[Lemma 1.21]{Willem}. It is easy to see that
 $$\int_{\mathbb{R}^3}|v_n(x)|^{\frac{16}{3}}dx=\int_{\mathbb{R}^3}|\tilde{v}_n(t_nx)|^{\frac{16}{3}}dx=t_n^{-3}\int_{\mathbb{R}^3}|\tilde{v}_n(x)|^{\frac{16}{3}}dx \rightarrow 0.$$
Similar to \eqref{eq02-12}, we get $\int_{\mathbb{R}^3}\tilde{H}(v_n) \rightarrow 0$. At the same time, $P(v_n)=0$ gives us
$$\int_{\mathbb{R}^3} |\nabla v_n|^2 +
5\int_{\mathbb{R}^3} |v_n|^2 |\nabla v_n|^2 =\frac{3}{2} \int_{\mathbb{R}^3} \tilde{H}(v_n) \rightarrow 0.$$
In view of Lemma \ref{lem2.2}, we obtain
$$0=P(v_n) \geq \frac{1}{2} \left(\int_{\mathbb{R}^3} |\nabla v_n|^2 +\int_{\mathbb{R}^3} |v_n|^2 |\nabla v_n|^2\right),$$
this is impossible. Therefore, $\tilde{\rho}>0$, which implies that there exists a sequence  $\{\tilde{y}_n\} \subset \mathbb{R}^3$ and $\tilde{v} \in H^1(\mathbb{R}^3)$ such that up to a subsequence $\tilde{v}_n(\cdot +\tilde{y}_n) \rightarrow \tilde{v}\neq 0$ a.e. in $\mathbb{R}^3$.

\textbf{Claim 2:}
$\limsup_{n \rightarrow \infty}s(\tilde{v}_n) \leq C.$

By Claim 1, we can let $\tilde{z}_n:=\tilde{v}_n(\cdot +\tilde{y}_n) \rightarrow \tilde{v} \neq 0$. Suppose that up to a subsequence $s(\tilde{v}_n) \rightarrow +\infty$ as $n \rightarrow \infty$. Hence, it follows from Lemma \ref{lem2.4}(3) that
$$s(\tilde{z}_n)=s(\tilde{v}_n(\cdot +\tilde{y}_n))=s(\tilde{v}_n) \rightarrow + \infty.$$
From $(h_4)$, Lemma \ref{lem2.1} and the Fatou's lemma, we get
$$\lim_{n \rightarrow +\infty} e^{-8s(\tilde{z}_n)}\int_{\mathbb{R}^3}H\left(e^{\frac{3}{2}s(\tilde{z}_n)}\tilde{z}_n\right)= +\infty.$$
As a result,
\begin{align*}
0&\leq e^{-5s(\tilde{z}_n)} I(s(\tilde{z}_n)*\tilde{z}_n) \\
&\leq  \frac{1}{2}e^{-3s(\tilde{z}_n)} \int_{\mathbb{R}^3} |\nabla \tilde{z}_n|^2 + \int_{\mathbb{R}^3} |\tilde{z}_n|^2 |\nabla \tilde{z}_n|^2 -e^{-8s(\tilde{z}_n)}\int_{\mathbb{R}^3}H\left(e^{\frac{3}{2}s(\tilde{z}_n)}\tilde{z}_n\right)  \\
&\rightarrow -\infty,
\end{align*}
as $n \rightarrow \infty$, a contradiction. The proof of Claim 2 is complete.

It is easy to see that
$$\int_{\mathbb{R}^3} |\nabla (s(\tilde{v}_n)*v_n)|^2 \leq C,$$
 $$\int_{\mathbb{R}^3} |s(\tilde{v}_n)*v_n|^2 |\nabla (s(\tilde{v}_n)*v_n)|^2 \leq C,$$
$$\int_{\mathbb{R}^3}H(s(\tilde{v}_n)*v_n) \leq C.$$
As a result,
$$m(a) \leq m(a_n) +o_n(1),$$
to wit,
$$m(a) \leq \liminf_{n \rightarrow \infty}m(a_n).$$
\end{proof}

\begin{lemma}\label{lem6.1}
Suppose $(h_1)-(h_7)$ hold, then $m(a) \rightarrow 0^+$ as $a \rightarrow +\infty$.
\end{lemma}

\begin{proof}
Fix $u \in \mathcal{S}'(1) \cap L^{\infty}(\mathbb{R}^3)$, set $u_a:=au \in \mathcal{S}'(a)$ for any $a>1$. By Lemma \ref{lem2.4}(1), there exists a unique $s_a \in \mathbb{R}$ such that $s_a*u_a \in \mathcal{P}'(a)$. And by means of Lemma \ref{lem2.1} and Remark \ref{rem3.5}, we know that
\begin{align*}
0<&\; m(a)\leq I(s_a*u_a)  \\
=&\ \frac{1}{2}e^{2s_a} \int_{\mathbb{R}^3} |\nabla u_a|^2 + e^{5s_a}\int_{\mathbb{R}^3} |u_a|^2 |\nabla u_a|^2 -e^{-3s_a}\int_{\mathbb{R}^3}  H\left(e^{\frac{3}{2}s_a}u_a\right)   \\
\leq&\ \frac{1}{2}a^2e^{2s_a} \int_{\mathbb{R}^3} |\nabla u|^2 + a^4e^{5s_a}\int_{\mathbb{R}^3} |u|^2 |\nabla u|^2.
\end{align*}
Next, we prove that
$$\lim_{a \rightarrow + \infty}ae^{s_a}=\lim_{a \rightarrow + \infty}ae^{\frac{5}{4}s_a}=0.$$
Indeed, by $s_a*u_a \in \mathcal{P}'(a)$, we see that
$$a^2e^{2s_a} \int_{\mathbb{R}^3} |\nabla u|^2 + 5a^4e^{5s_a}\int_{\mathbb{R}^3} |u|^2 |\nabla u|^2 =\frac{3}{2}e^{-3s_a}\int_{\mathbb{R}^3}  \tilde{H}\left(e^{\frac{3}{2}s_a}au\right).$$
And it follows from $(h'_5)$ that
\begin{align*}
&\ a^{-2}e^{-3s_a} \int_{\mathbb{R}^3} |\nabla u|^2 + 5\int_{\mathbb{R}^3} |u|^2 |\nabla u|^2  \\
=&\ \frac{3}{2}a^{-4}e^{-8s_a}\int_{\mathbb{R}^3}  \tilde{H}\left(e^{\frac{3}{2}s_a}au\right)  \\
=&\ \frac{3}{2}a^\frac{4}{3}\int_{\mathbb{R}^3}  \frac{\tilde{H}\left(e^{\frac{3}{2}s_a}au\right)}{|e^{\frac{3}{2}s_a}au|^{\frac{16}{3}}} |u|^{\frac{16}{3}} \\
\geq &\ \frac{3}{2}a^\frac{4}{3} \frac{\tilde{H}\left(e^{\frac{3}{2}s_a}a\frac{\|u\|_\infty}{2}\right)}{|e^{\frac{3}{2}s_a}a\frac{\|u\|_\infty}{2}|^{\frac{16}{3}}} \int_{|u|\geq \frac{\|u\|_\infty}{2}}  |u|^{\frac{16}{3}}.
\end{align*}
If $\limsup_{a \rightarrow + \infty}ae^{\frac{3}{2}s_a}=+\infty$ \,or\, $\limsup_{a \rightarrow + \infty}ae^{\frac{3}{2}s_a}=C_1>0$, we can derive a contradiction. As a result,
$$\lim_{a \rightarrow + \infty}ae^{\frac{3}{2}s_a}=0 \;\;\mbox{and}\;\; \lim_{a \rightarrow + \infty}e^{s_a}=0.$$
Assume $\limsup_{a \rightarrow + \infty}ae^{s_a}=C_2>0$, taking into account $(h_6)$, $(h_7)$ and Lemma \ref{lem2.1}, we can deduce that for $M>C_2^{-4}\frac{\int_{\mathbb{R}^3} |\nabla u|^2}{\int_{\mathbb{R}^3} |u|^{6}}>0$, there exists a $\delta>0$ such that as $|t|\leq \delta$, we have
$$\tilde{H}(t)\geq \frac{10}{3}H(t)\geq \frac{5}{9}h(t)t \geq M|t|^{6}.$$
Therefore, for $a$ satisfying $|e^{\frac{3}{2}s_a}au|\leq \delta$, there holds
\begin{align*}
&\ a^{-3}e^{-3s_a} \int_{\mathbb{R}^3} |\nabla u|^2 + 5a^{-1}\int_{\mathbb{R}^3} |u|^2 |\nabla u|^2
= \frac{3}{2}a^{-5}e^{-8s_a}\int_{\mathbb{R}^3}  \tilde{H}\left(e^{\frac{3}{2}s_a}au\right)  \\
\geq&\ \frac{3}{2}a^{-5}e^{-8s_a} M|e^{\frac{3}{2}s_a}a|^{6} \int_{\mathbb{R}^3}|u|^{6} \\
=&\ \frac{3}{2}Mae^{s_a}\int_{\mathbb{R}^3}|u|^{6},
\end{align*}
which implies
$$C_2^{-3}\int_{\mathbb{R}^3} |\nabla u|^2 \geq \frac{3}{2}MC_2\int_{\mathbb{R}^3}|u|^{6}.$$
Therefore,
$$\frac{2}{3}C_2^{-4}\frac{\int_{\mathbb{R}^3} |\nabla u|^2}{\int_{\mathbb{R}^3} |u|^{6}}\geq M.$$
This is impossible. And if $\limsup_{a \rightarrow + \infty}ae^{s_a}=+\infty$, we can derive a contradiction. Thereby,
$$\lim_{a \rightarrow + \infty}ae^{s_a}=0.$$
And it follows from $\lim_{a \rightarrow + \infty}e^{s_a}=0$ that
$$\lim_{a \rightarrow + \infty}ae^{\frac{5}{4}s_a}=0.$$
It is clear that $m(a) \rightarrow 0^+$ as $a \rightarrow +\infty$.

\end{proof}

\begin{lemma}\label{lem6.2}
    Assume $h$ satisfies $(h_1)-(h_4)$ and $(h_5')$, then $m(a) \rightarrow +\infty$ as $a \rightarrow 0^+$.
\end{lemma}
\begin{proof}
It is sufficient to show that for any sequence $\{u_n\} \subset H^1(\mathbb{R}^3)\setminus \{0\}$ such that $\lim_{n \rightarrow \infty}\|u_n\|_2=0$ and $P(u_n)=0$, one has $I(u_n) \rightarrow +\infty$ as $n \rightarrow \infty$.

Take $s_n$ such that
$$e^{-2s_n} \int_{\mathbb{R}^3} |\nabla u_n|^2 + e^{-5s_n}\int_{\mathbb{R}^3} |u_n|^2 |\nabla u_n|^2=1.$$
Set $v_n:=(-s_n)*u_n$, it is clear that
$$ \int_{\mathbb{R}^3} |\nabla v_n|^2+\int_{\mathbb{R}^3} |v_n|^2 |\nabla v_n|^2=1,\;  \int_{\mathbb{R}^3} |v_n|^2= \int_{\mathbb{R}^3} |u_n|^2 \rightarrow 0\;\; \mbox{and}\;\; s(v_n)=s_n.$$
And by \eqref{GNI}, we obtain $\int_{\mathbb{R}^3} |v_n|^{\frac{16}{3}} \rightarrow 0$. Further, using the similar approach to \eqref{eq02-12}, we conclude that
\begin{equation*}
    \lim_{n \rightarrow \infty} e^{-3s}\int_{\mathbb{R}^3} H\Big(e^{\frac{3}{2}s}v_n\Big)=0 \;\; \mbox{for} \;\; \mbox{any}\;\; s>0.
\end{equation*}
Since $P(s(v_n)*v_n)=P(u_n)=0$, we derive that for $s>0$,
\begin{align*}
    &\ I(s(v_n)*v_n) \geq I(s*v_n) \\
      =&\ \frac{1}{2}e^{2s} \int_{\mathbb{R}^3} |\nabla v_n|^2 + e^{5s}\int_{\mathbb{R}^3} |v_n|^2 |\nabla v_n|^2 -e^{-3s}\int_{\mathbb{R}^3}  H\left(e^{\frac{3}{2}s}v_n\right) \\
      \geq&\ \frac{1}{2}e^{2s} \left(\int_{\mathbb{R}^3} |\nabla v_n|^2 +\int_{\mathbb{R}^3} |v_n|^2 |\nabla v_n|^2\right) +o_n(1) \\
      =&\ \frac{1}{2}e^{2s} +o_n(1).
\end{align*}
As $s>0$ is arbitrary, it is clear that $I(u_n)=I(s(v_n)*v_n) \rightarrow +\infty$ as $n \rightarrow \infty$.
\end{proof}

\begin{proof}[Proof of Theorem \ref{Thm1.5}]
Owing to the fact $\{u \in \mathcal{S}'(a)\, |\, \Theta(u,\lambda)=0\} \subset \mathcal{P}'(a)$, we conclude $m(a) \leq \sigma(a).$ Together with Remark \ref{rem3.5}, Lemma \ref{lem6.3}$-$Lemma \ref{lem6.2}, we can complete the proof of Theorem \ref{Thm1.5}.
\end{proof}

\section*{Acknowledgements}
This paper is partially supported by the National Natural Science Foundation of China (No. 11571200) and the Natural Science Foundation of Shandong Province (No. ZR2021MA062).

\section*{Conflict of Interest Statement}
The authors declare that they have no conflict of interest.

\section*{Data Availability Statement}
The manuscript has no associated data.

\end{document}